\documentclass[reqno,12pt,letterpaper]{amsart}
\usepackage{amsmath,amssymb,amsthm,graphicx,mathrsfs,url}\usepackage[usenames,dvipsnames]{color}
\usepackage[colorlinks=true,linkcolor=Red,citecolor=Green]{hyperref}
\usepackage{amsxtra}
\usepackage{subcaption}

\def\arXiv#1{\href{http://arxiv.org/abs/#1}{arXiv:#1}}

\def\?[#1]{\textbf{[#1]}\marginpar{\Large{\textbf{??}}}}
\def\smallsection#1{\smallskip\noindent\textbf{#1}.}
\let\epsilon=\varepsilon 

\newtheorem{theo}{Theorem}
\newtheorem{prop}{Proposition}[section]	
\newtheorem{defi}[prop]{Definition}

\newtheorem{ass}{Assumption}
\newtheorem{lemm}[prop]{Lemma}
\newtheorem{corr}[prop]{Corollary}
\newtheorem{rem}{Remark}
\numberwithin{equation}{section}

\DeclareMathOperator{\tr}{tr}

\def\indic{\operatorname{1\hskip-2.75pt\relax l}}
\usepackage{scalerel}

\newcommand\reallywidehat[1]{\arraycolsep=0pt\relax%
\begin{array}{c}
\stretchto{
  \scaleto{
    \scalerel*[\widthof{\ensuremath{#1}}]{\kern-.5pt\bigwedge\kern-.5pt}
    {\rule[-\textheight/2]{1ex}{\textheight}} 
  }{\textheight} %
}{0.5ex}\\           
#1\\                 
\rule{-1ex}{0ex}
\end{array}
}

\title[Infinite-dimensional bilinear and stochastic balanced truncation]{Infinite-dimensional bilinear and stochastic balanced truncation with error bounds}

\author{Simon Becker}
\email{simon.becker@damtp.cam.ac.uk}
\address{DAMTP, University of Cambridge, Wilberforce Road, Cambridge CB3 0WA, UK}
\author{Carsten Hartmann}
\email{carsten.hartmann@b-tu.de}
\address{Institute for Mathematics, BTU Cottbus-Senftenberg, 03046 Cottbus, Germany}

\begin{document}

\begin{abstract}
Along the ideas of Curtain and Glover \cite{CG}, we extend the balanced truncation method for infinite-dimensional linear systems to arbitrary-dimensional bilinear and stochastic systems. In particular, we apply Hilbert space techniques used in many-body quantum mechanics to establish new fully explicit error bounds for the truncated system and prove convergence results. The functional analytic setting allows us to obtain mixed Hardy space error bounds for both finite-and infinite-dimensional systems, and it is then applied to the model reduction of stochastic evolution  equations driven by Wiener noise.        
\end{abstract}

\maketitle


\section{Introduction}
Model reduction of bilinear systems has become a major field of research, partly triggered by applications in optimal control and the advancement of iterative numerical methods for solving large-scale matrix equations. High-dimensional bilinear systems often appear in connection with semi-discretised controlled partial differential equations or stochastic (partial) differential equations with multiplicative noise. A popular class of model reduction methods that is well-established in the field of linear systems theory is based on first transforming the system to a form in which highly controllable states are highly observable and vice versa (``balancing''), and then eliminating the least controllable and observable states. For finite-dimensional linear systems, balanced truncation and residualisation (a.k.a.~\emph{singular perturbation approximation}) feature computable error bounds and are known to preserve important system properties, such as stability or passivity \cite{Gl84}; see also \cite{Ant05} and references therein. For a generalisation of (linear) balanced truncation to infinite-dimensional systems, see \cite{CG,GO}. 

For bilinear systems, no such elaborate theory as in the linear case is available, in particular approximation error bounds for the reduced system are not known. The purpose of this paper therefore is to extend balanced truncation to bilinear and stochastic evolution equations,  
specifically, to establish convergence results and prove explicit truncation error bounds for the bilinear and stochastic systems. For finite-dimensional systems our framework coincides with the established theory for bilinear and stochastic systems as studied in \cite{BD,ZL}, and references therein. 
We start by introducing a function space setting that allows us to define bilinear balanced truncation in arbitrary (separable) Hilbert spaces which extends the finite-dimensional theory. However, instead of just extending the finite-dimensional theory to infinite dimensions, we harness the functional analytic machinery available in infinite dimensions to obtain new explicit error bounds for finite-dimensional systems as well. 

The \emph{figure of merit} in our analysis is a Hankel-type operator acting between certain function spaces which are ubiquitous in many-body quantum mechanics and within this theory called \emph{Fock spaces}. 
We show that under mild assumptions on the dynamics, the Hankel operator is a Hilbert-Schmidt or even trace class operator. The key idea is that the algebraic structure of the  Fock space, that is a direct sum of tensor products of copies of Hilbert spaces, mimics the nested Volterra kernels representing the bilinear system. This allows us to perform an analysis of the singular value decomposition of this operator along the lines of the linear theory developed by Curtain and Glover \cite{CG}. 
For more recent treatments of infinite-dimensional \emph{linear systems} we refer to \cite{GO,ReSe,Si}. For applications of the bilinear method to finite-dimensional open quantum systems and Fokker-Planck equations we refer to \cite{HSZ} and \cite{SHSS}.

The article is structured as follows: The rest of the introduction is devoted to fix the notation that is used throughout the article and to state the main results. Section \ref{sec:bilinearBT} introduces the concept of balancing based on observability and controllability (or reachability) properties of bilinear systems, which is then used in Section \ref{sec:Hankel} to define the Fock space-valued Hankel operator and study properties of its approximants. The global error bounds for the finite-rank approximation based on the singular value decomposition of the Hankel operator are given in Section \ref{sec:errorBound}. Finally, in Section \ref{sec:spdes} we discuss applications of the aforementioned results to the model reduction of stochastic evolution equations driven by multiplicative L\'evy noise. The article contains two appendices. The first one records a technical lemma stating the Volterra series representation of the solution to infinite-dimensional bilinear systems. The second appendix provides more background on how to compute the error bounds found in this article.

\subsection*{Set-up and main results}

Let $X$ be a separable Hilbert space and $A:D(A) \subset X \rightarrow X$ the generator of an exponentially stable $C_0$-semigroup $(T(t))_{t \ge 0}$ of bounded operators, i.e. a strongly continuous semigroup that satisfies $\left\lVert T(t)  \right\rVert \le M e^{-\nu t}$ for some $\nu>0$ and $M \ge 1.$

For exponentially stable semigroups generated by $A$, bounded operators $N_i \in \mathcal L (X)$, $B \in \mathcal L(\mathbb{R}^n,X)$, an initial state $\varphi_0 \in X$, and control functions $u=(u_1,...,u_n) \in L^{2}((0,T),\mathbb{R}^n)$, we study bilinear evolution equations on $X$ of the following type
\begin{equation}
 \begin{split}
\label{evoleq}
\varphi'(t) &= A \varphi(t) + \sum_{i=1}^n N_i \varphi(t) u_i(t) + Bu(t), \text{ for }t \in (0,T) \text{ such that } \\
\varphi(0) &=\varphi_0.
 \end{split}
\end{equation}
It follows from standard fixed-point arguments that such equations always have unique mild solutions \cite[Proposition 5.3]{LiYo} $\varphi \in C([0,T],X)$ that satisfy  
\begin{equation}
\label{mild solution}
\varphi(t) = T(t)\varphi_0 + \int_{0}^{t} T(t-s)\left(\sum_{i=1}^n u_i(s)N_i \varphi(s)+Bu(s) \right)  \ ds.
\end{equation}
Let $\Gamma:=\sqrt{ \sum_{i=1}^n \left\lVert   N_iN_i^* \right\rVert}$ and assume that $M^{2} \Gamma^{2} ( 2 \nu )^{-1}<1.$ We then introduce the observability $\mathscr{O}= W^* W $ and reachability gramian $\mathscr{P} = R R^*$ for equation \eqref{evoleq} in Definition \ref{gramians}. 
The gramians we define coincide for finite-dimensional system spaces $X\simeq \mathbb R^k,$ and control $B \in \mathcal L(\mathbb R^n,\mathbb R^k)$ and observation $C\in \mathcal L(\mathbb R^k,\mathbb R^m)$ matrices of suitable size with the gramians introduced in \cite{DIR}, see also \cite[(6) and (7)]{ZL}. More precisely, if $X$ is finite-dimensional then the reachability gramian $\mathscr P$ is defined by
\begin{equation*}
\begin{split}
&P_1(t_1) = e^{At_1} B, \\ 
&P_i(t_1,..,t_i) = e^{A t_1}\left( N_1 P_{i-1} \ N_2 P_{i-1} \ \cdots N_n P_{i-1} \right)(t_2,..,t_i), \ i\ge 2\\
&\mathscr P = \sum_{i=1}^{\infty} \int_{(0,\infty)^i} P_i(t_1,,.,t_i) P_i^T(t_1,..,t_i) \ dt 
\end{split}
\end{equation*}
and the observability gramian $\mathscr O$ by
\begin{equation*}
\begin{split}
&Q_1(t_1) = Ce^{At_1} , \\ 
&Q_i(t_1,..,t_i) = \left( Q_{i-1} N_1  \ Q_{i-1} N_2 \ \cdots Q_{i-1} N_n  \right)(t_2,..,t_i)e^{A t_1}, \ i\ge 2\\
&\mathscr O = \sum_{i=1}^{\infty} \int_{(0,\infty)^i} Q_i^T(t_1,,.,t_i) Q_i(t_1,..,t_i) \ dt.
\end{split}
\end{equation*}
The condition $M^{2} \Gamma^{2} ( 2 \nu )^{-1}<1$, stated in the beginning of this paragraph, appears naturally to ensure the existence of the two gramians. To see this, consider for example the reachability gramian for which we find \cite[Theorem 2]{ZL}
\begin{equation*}
\left\lVert \mathscr P \right\rVert 
\le \sum_{i=1}^{\infty} \int_{(0,\infty)^i} \left\lVert P_i(t_1,,.,t_i) P_i^T(t_1,..,t_i)\right\rVert dt \le \frac{\left\lVert BB^T \right\rVert}{\Gamma^2}\sum_{i=1}^{\infty} \left(\frac{M^2 \Gamma^2}{2\nu}\right)^i
\end{equation*}
which is summable if $M^{2} \Gamma^{2} ( 2 \nu )^{-1}<1.$

For general bilinear and stochastic systems the gramians will be decomposed, as indicated above, by an observability $W$ and reachability map $R$ that are \emph{explicitly} constructed in Section \ref{sec:Hankel}. Although there are infinitely many possible decompositions of the gramians, our analysis relies on having an explicit decomposition. In particular, $R$ will be chosen as a map from a Fock space to the Hilbert space $X$ on which the dynamics of the system is defined whereas $W$ maps the Hilbert space $X$ back into a Fock space again. The \emph{Hankel operator} is then defined as $H= W  R$ and maps between Fock spaces.
From the Hankel operator construction we obtain two immediate corollaries: \newline
The full Lyapunov equations for bilinear or stochastic systems are known to be notoriously difficult to solve. It is therefore computationally more convenient \cite{B1} to compute a $k$-th order truncation of the gramians which we introduce in Definition \ref{trungram}. 
Our first result implies exponentially fast convergence of the balanced singular values calculated from the truncated gramians to the balanced singular values obtained from the full gramians $\mathscr O$ and $\mathscr P$:
\begin{prop}
\label{singularvalueconv}
Let $(\sigma_i)_{i \in \mathbb N}$ denote the balanced singular values $\sigma_i:= \sqrt{\lambda_i(\mathscr{O} \mathscr{P})}$ and $(\sigma^{k}_i)_{i \in \mathbb N}$ the singular values of the $k$-th order truncated gramians. The Hankel operator $H^k$ computed from the $k$-th order truncated gramians converges in Hilbert-Schmidt norm to $H$ and for all $i \in \mathbb{N}$
\begin{equation*}
\left\lvert \sigma_{i}-\sigma^k_{i} \right\rvert = \mathcal O \left(  \left(\underbrace{M^{2} \Gamma^{2} ( 2 \nu )^{-1}}_{<1} \right)^k \right) \text{ and } \left\lvert \left\lVert \sigma \right\rVert_{\ell^2}-\left\lVert \sigma^k \right\rVert_{\ell^2} \right\rvert = \mathcal O \left(  \left(\underbrace{M^{2} \Gamma^{2} ( 2 \nu )^{-1}}_{<1} \right)^k \right).
\end{equation*}
\end{prop}

Although our framework includes infinite-dimensional systems, such systems are numerically approximated by finite-dimensional systems. 

We therefore discuss now a convergence result for systems that can be approximated by projections onto smaller subspaces.
Let $V_1 \subset V_2 \subset ...\subset X$ be a nested sequence of closed vector spaces of arbitrary dimension such that $\overline{\bigcup_{i \in \mathbb{N}} V_i }=X$
for which we assume that $V_i$ is an invariant subspace of both $T(t)$ and $N$. In this case, $V_i$ is also an invariant subspace of the generator $A$ of the semigroup \cite[Chapter $2,$ Section $2.3$]{EN}, and we can consider the restriction of \eqref{evoleq} to $V_i$ \footnote{$P_{V_i}$ is the orthogonal projection on the closed space $V_i$.}
\begin{equation*}
 \begin{split}
\varphi_{V_i}'(t) &= A \varphi_{V_i}(t) + \sum_{i=1}^n u_i(t) N_i \varphi_{V_i}(t)  + P_{V_i}Bu(t), \text{ for }t \in (0,T) \\
\varphi(0) &= P_{V_i}(\varphi_0).
 \end{split}
\end{equation*}

\begin{prop}
\label{finapprox}
Let $H_{V_i}$ be the Hankel operator of the system restricted to $V_i$. If the observability map $W$ is a Hilbert-Schmidt operator, then the Hankel operator $H_{V_i}$ converges in nuclear (trace) norm to $H.$ If $W$ is only assumed to be bounded, then the convergence of Hankel operators is still in Hilbert-Schmidt norm. 
\end{prop}
Sufficient conditions for $W$ to be a Hilbert-Schmidt operator will be presented in Lemma \ref{traceclasslemma}.
Norm convergence of Hankel operators implies convergence of its singular values and so the convergence of Hankel singular values holds also under the assumptions of Proposition \ref{finapprox}.

We then turn to global error bounds for bilinear systems: For linear systems, the existence of a Hardy space $\mathscr{H}^{\infty}$ error bound is well-known and a major theoretical justification of the linear balanced truncation method both in theory and practice. That is, the difference of the transfer function for the full and reduced system in $\mathscr{H}^{\infty}$ norm is controlled by the difference of the Hankel singular values that are discarded in the reduction step. To the best of our knowledge, there is no such bound for bilinear systems and we are only of aware of two recent results in that direction \cite{Re} and \cite{Re18}.  

In \cite{BD2} a family of \emph{transfer functions} $(G_k)_{k \in \mathbb N_0}$ for bilinear systems was introduced. We consider the difference of these transfer functions for two systems and write $\Delta(G_k)$ for the difference of transfer functions and $\Delta(H)$ for the difference of Hankel operators. 
In terms of these two quantities we obtain an error bound that extends the folklore bound for linear systems to the bilinear case:

\begin{theo}
\label{singestim}
Consider two bilinear systems that both satisfy the stability condition $M^2\Gamma^2(2\nu)^{-1}<1$ with the same finite-dimensional input space $\mathbb R^n$ and output space $\mathcal H \simeq \mathbb R^m.$\footnote{ We freely identify $\mathcal H$ with $\mathbb R^m$ in the sequel when we assume that they are isomorphic. }
The difference of the transfer functions of the two systems $\Delta(G_k)$ in mixed $\mathscr{H}^{\infty}$-$\mathscr{H}^2$ Hardy norms, defined in \eqref{eq: mix}, is bounded by
\begin{equation*}
\begin{split}
& \sum_{k=1}^{\infty} \left( \left\lVert \Delta(G_{2k-2}) \right\rVert_{\mathscr H^{\infty}_k \mathscr H^2_{2k-2}}+\left\lVert \Delta(G_{2k-1}) \right\rVert_{\mathscr H^{\infty}_k \mathscr H^2_{2k-1}} \right) \le 4 \left\lVert \Delta(H) \right\rVert_{\operatorname{TC}}.
\end{split}
\end{equation*}
The trace distance of the Hankel operators can be explicitly evaluated using the composite error system, see Appendix \ref{sec:comperrsys}, and does not require a direct computation of Hankel operators.  
\end{theo}

The proof of Theorem \ref{singestim} is done by extending the framework from the linear balancing theory and extends the $2 \left\lVert \Delta(H) \right\rVert_{\operatorname{TC}}$ bound on the $\mathscr{H}^{\infty}$ norm of the transfer function for linear equations to a somewhat different, but still explicitly computable bound for bilinear systems. 
From the Hankel estimates we then obtain an explicit error bound on the dynamics for two systems with initial condition zero:
\begin{theo}
\label{dyncorr}
Consider two bilinear systems that both satisfy the stability condition $M^2\Gamma^2(2\nu)^{-1}<1$ with the same finite-dimensional input space $\mathbb R^n$ and output space $\mathcal H \simeq \mathbb R^m$. 
Let $\Delta(C\varphi(t))$ be the difference of the outputs of the two systems. 
For control functions $u \in L^{\infty}((0,\infty),\mathbb R^n) \cap L^2((0,\infty),\mathbb R^n)$ such that $\left\lVert u \right\rVert_{L^{2}((0,\infty),(\mathbb R^n, \left\lVert \bullet \right\rVert_{\infty}))} < \min \left(\frac{1}{\sqrt{n}}, \frac{\sqrt{2\nu}}{M \Xi} \right)$ with $\Xi:=\sum_{i=1}^{n}\|N_{i}\|$ and initial conditions zero it follows that
\begin{equation*}
\sup_{t \in (0, \infty)} \left\lVert \Delta(C\varphi(t)) \right\rVert_{\mathbb R^m} \le 4\sqrt{n} \left\lVert \Delta(H) \right\rVert_{\operatorname{TC}} \left\lVert u \right\rVert_{L^{\infty}((0,\infty),(\mathbb R^n, \left\lVert \bullet \right\rVert_{\infty}))}.
\end{equation*}
As mentioned in Theorem \ref{singestim} the trace distance of the Hankel operators can be explicitly evaluated using the composite error system, see Appendix \ref{sec:comperrsys}, and does not require a direct computation of Hankel operators. 
\end{theo}

As an application of the theoretical results, we discuss generalised stochastic balanced truncation of stochastic (partial) differential equations in Section \ref{sec:spdes}. 
The links between bilinear balanced truncation and stochastic balanced truncation are well-known for finite-dimensional systems driven by Wiener noise (see e.g. \cite{BD}). In Section \ref{sec:spdes}, we extend the Hankel operator methods to the finite-dimensional stochastic systems discussed in \cite{BD3} and \cite{BR15}, but our methods also cover a large class of infinite-dimensional stochastic systems as well. 
By pursuing an approach similar to the linear setting, we obtain an error bound on the expected output in terms of the Hankel singular values:
\begin{prop}
\label{theo2}
Consider two stochastic systems with the same finite-dimensional input space $\mathbb R^n$ and output space $\mathcal H \simeq \mathbb R^m$. Let $u \in L^p((0,\infty),\mathbb R^n)$ for $p \in [1,\infty]$ be a deterministic control and let $\Phi$ and $\widetilde{\Phi}$ be the stochastic flows of each respective system. The two stochastic flows shall be exponentially stable in mean square sense and define $ C_b$-Markov semigroups. The difference $\Delta(CY)$ of processes $Y$ defined in \eqref{eq:DuHamTra} with initial conditions zero satisfies then
\begin{equation*}
 \left\lVert  \mathbb E   \Delta(CY_{\bullet}(u)) \right\rVert_{L^p((0,\infty),\mathbb R^m)} \le 2 \left\lVert \Delta (H) \right\rVert_{\operatorname{TC}} \left\lVert u \right\rVert_{L^{p}((0,\infty),\mathbb R^n)}.
\end{equation*}
The trace distance of the Hankel operators can be explicitly evaluated using the composite error system, see Appendix \ref{sec:comperrsys}.
\end{prop}
It was first shown in \cite[Example II.2]{BD3} that stochastic systems do not obey error bounds that are linear in the truncated singular values as one has for example from the theory of linear balanced truncation.
Yet, the following result can be obtained by arguing along the lines of the bilinear framework:
\begin{theo}
\label{theo3}
Consider two stochastic systems with the same finite-dimensional input space $\mathbb R^n$ and output space $\mathcal H \simeq \mathbb R^m$ such that the respective stochastic flows $\Phi$ and $\widetilde{\Phi}$ are independent. The two stochastic flows shall be exponentially stable in mean square sense and define $ C_b$-Markov semigroups. The difference $\Delta(CY)$ of processes $Y$ defined in \eqref{eq:DuHamTra} with zero initial conditions satisfies
\begin{equation}
\label{eq:estimate}
\sup_{t \in(0,\infty)} \mathbb E \left\lVert \Delta(CY_t(u)) \right\rVert_{\mathbb R^m} \le 2  \left\lVert \Delta(  H) \right\rVert_{\operatorname{TC}} \left\lVert u\right\rVert_{\mathcal H_2^{(0,\infty)}(\mathbb R^n)}
\end{equation} 
with controls in the Banach space $\left(\mathcal H_{2}^{(0,\infty)}(\mathbb R^n), \sup_{t \in (0,\infty)}\left( \mathbb E \left(   \left\lVert u(t) \right\rVert_{\mathbb R^n}\right)^2\right)^{1/2}\right).$ The trace distance of the Hankel operators can be explicitly evaluated using the composite error system, see Appendix \ref{sec:comperrsys}.
\end{theo}

\subsection*{Finite-dimensional intermezzo and relation to balanced truncation}
Hitherto, stochastic and bilinear balanced truncation have only been considered for finite-dimensional systems and so we devote a few preliminary remarks towards this setting.
When applying for example balanced truncation to finite-dimensional systems one computes the observability and reachability gramians $\mathscr O $ and $\mathscr P$ from the Lyapunov equations and decomposes these symmetric positive-definite matrices into some other (non-unique) matrices $\mathscr O=K^*K$ and $\mathscr P = VV^*.$ In the next step, a singular value decomposition of the matrix $ KV$ is computed. The singular values of this matrix $KV$ are just the square-roots of the eigenvalues of the product of the gramians $\sigma_j:=\sqrt{\lambda_j(\mathscr O \mathscr P)}$ independent of the particular form of $K$ and $V$ (zero is not counted as a singular value here). 

By discarding a certain number of "small" singular values of $KV$, one can reduce the order of the system by applying for example the balancing transformations, see \cite[Proposition $2$]{ZL}.
A paradigm of such a decomposition $KV$, where $K$ and $V$ are not matrices but operators, is the Hankel operator $H$. 
Yet most importantly, all such decompositions of the gramians are equivalent \cite[Theorem $5.1$]{ReSe}. That is, there are unitary transformations $U_1: \overline{\operatorname{ran}}(H) \rightarrow \overline{\operatorname{ran}}(KV)$ and $U_2: \operatorname{ker}(H)^{\perp} \rightarrow \operatorname{ker}(KV)^{\perp}$ such that any decomposition $KV\vert_{\operatorname{ker}(KV)^{\perp}} $ of the gramians is equivalent to the Hankel operator studied in this paper  $H \vert_{\operatorname{ker}(H)^{\perp}} = U_1^* \ KV\vert_{\operatorname{ker}(KV)^{\perp}}  U_2.$ This makes our results on error bounds widely applicable since the Hankel decomposition is as good as any other decomposition.  

This is because evaluating the trace norm of the difference of Hankel operators appearing in our error bound, one only has to compute the gramians of the composite system and not the actual Hankel operators, see the explanation given in Appendix \ref{sec:comperrsys}. In particular, the respective gramians of the composite system can be computed for example directly from the Lyapunov equations of the composite error system.
\subsection*{Notation}
The space of bounded linear operators between Banach spaces $X,Y$ is denoted by $\mathcal L(X,Y)$ and just by $\mathcal L(X)$ if $X=Y.$ The operator norm of a bounded operator $T \in \mathcal L(X,Y)$ is written as $\left\lVert T \right\rVert$. The trace class operators from $X$ to $Y$ are denoted by $\operatorname{TC}(X,Y)$ and the Hilbert-Schmidt operators by $\operatorname{HS}(X,Y).$ In particular, we recall that for a linear trace class operator $T \in \operatorname{TC}(X,Y)$, where $X$ and $Y$ are separable Hilbert spaces, the trace norm is given by the following supremum over orthonormal systems of basis vectors (ONB),
\begin{equation}
\label{tracenorm}
\left\lVert T \right\rVert_{\operatorname{TC}}=\sup \left\{ \sum_{n \in \mathbb{N}} \left\lvert \langle f_n,T e_n  \rangle_Y \right\rvert ; (e_n) \text{ ONB of } X \text{ and } (f_n) \text{ ONB of } Y \right\}.
\end{equation}
We write $\partial B_X(1)$ for the unit sphere of a Banach space $X$ and say that $g=\mathcal O (f)$ if there is $C>0$ such that $\left\lVert g \right\rVert \le C \left\lVert f \right\rVert.$ In order not to specify the constant $C$, we also write $\left\lVert g \right\rVert \lesssim  \left\lVert f \right\rVert.$ The indicator function of an interval $I$ is denoted by $\indic_I.$
The domain of unbounded operators $A$ is denoted by $D(A).$

Let $\mathnormal H$ be a separable Hilbert space. For the n-fold Hilbert space tensor product of a Hilbert space $\mathnormal H$ we write $\mathnormal H^{\otimes n}:=\mathnormal H \otimes ...\otimes \mathnormal H.$ To define the Hankel operator we require a decomposition of the positive gramians.  For this purpose, we introduce the Fock space $F^n(\mathnormal H)$ of $\mathnormal H$-valued functions $F^n(\mathnormal H):=\bigoplus_{k=1}^{\infty}F_k^n(\mathnormal H)$ where $F_k^n(\mathnormal H):=L^2((0,\infty)^k,\mathnormal H \otimes (\mathbb R^{n})^{\otimes (k-1)})$ and $F_0^n(\mathnormal H):=\mathnormal H.$

Thus, elements of the Fock space $F^n$ are sequences taking values in $F^n_k.$ 

Let $\mathbb C_{+}$ be the right complex half-plane, then we define the $\mathnormal H$-valued Hardy spaces $\mathscr H^2$ and $\mathscr H^{\infty}$ of multivariable holomorphic functions $F:\mathbb C_{+}^k \rightarrow \mathnormal H$ with finite norms
\begin{equation*}
\left\lVert F \right\rVert_{\mathscr H^2}:=\sup_{x \in \mathbb R^{k}} \frac{1}{(2\pi)^{k/2}} \left( \int_{(0,\infty)^{k}} \left\lVert F(x+iy) \right\rVert_{\mathnormal H}^2 \ dy\right)^{\frac{1}{2}} \text{ and } \left\lVert F \right\rVert_{\mathscr H^{\infty}}:=\sup_{z \in \mathbb C_{+}^{k}}  \left\lVert F(z) \right\rVert_{\mathnormal H},
\end{equation*} 
respectively. We also introduce mixed $L^1_iL^2_{k-1}$ and $\mathscr H^{\infty}_i \mathscr H^{2}_{k-1}$ norms which for $\mathnormal H$-valued functions $f:(0,\infty)^k \rightarrow  \mathnormal H$ and $g: \mathbb C_{+}^k \rightarrow \mathnormal H$ read
\begin{equation}
\label{eq: mix}
\begin{split}
&\left\lVert f \right\rVert_{L^1_iL^2_{k-1}(\mathnormal H)} = \int_0^{\infty} \left\lVert f(\bullet,..,\bullet,s_i, \bullet,..,\bullet) \right\rVert_{L^2((0,\infty)^{k-1},\mathnormal H)} \ ds_i \text{ and } \\
&\left\lVert g \right\rVert_{\mathscr H^{\infty}_i \mathscr H^{2}_{k-1}(\mathnormal H)}= \sup_{s_i \in \mathbb C_{+}} \left\lVert g(\bullet,..,\bullet,s_i, \bullet,..,\bullet) \right\rVert_{\mathscr H^2((0,\infty)^{k-1},\mathnormal H)}. 
\end{split}
\end{equation}
Finally, for $k$-variable functions $h$ we occassionally use the short notation
\begin{equation}
\label{eq:k}
h^{(i)}(s,t):=h(s_1,...,s_{i-1},t,s_{i},..,s_{k-1}).
\end{equation}

In Section \ref{sec:spdes}, the space $L^p_{\text{ad}}$ denotes the $L^p$ spaces of stochastic processes that are adapted to some given filtration and we introduce the notation $\Omega_I:=I \times \Omega$ where $I$ is some interval.

\section{The pillars of bilinear balanced truncation}\label{sec:bilinearBT}

We start with the definition of the gramians on $X$ which extend the standard definition on finite-dimensional spaces \cite[(6)~and~(7)]{ZL} to arbitrary separable Hilbert spaces. 
\subsection{Gramians}
Let  $\mathcal{H}$ be a separable Hilbert space and $C \in \mathcal L(X,\mathcal{H})$ the state-to-output (observation) operator. The space $\mathcal H$ is called the \emph{output space}. As we assume that there are $n$ control functions, the space $\mathbb R^n$ will be referred to as the \emph{input space}. Adopting the notation used in \eqref{evoleq} with strongly continuous semigroup $(T(t))$ generated by $A$, we then introduce the bilinear gramians for times $t_i \in (0,\infty)$:

\begin{defi} 
\label{gramians} 
Let $O_0(t_1):=CT(t_1)$. Then, for $i \ge 1$ and $y \in X$ define 
\begin{equation*}
O_{i}(t_1,..,t_{i+1})y:=CT(t_1) \sum_{n_1,...,n_{i}=1}^n \left(\prod_{l=2}^{i+1} \left(N_{n_{l-1}}T(t_l) \right)\right)y \otimes \left(\widehat{e}_{n_1} \otimes ...\otimes \widehat{e}_{n_{i}}\right)
\end{equation*}
with $\widehat{e}_i$ denoting the standard basis vectors of $\mathbb R^n.$

Let $M^2\Gamma^2(2\nu)^{-1} < 1 $, then the bounded operators $\mathscr{O}_k$ defined for $x,y \in X$ by
\begin{equation}
\label{eq:O}
\langle x, \mathscr{O}_ky \rangle_X := \int_{(0,\infty)^{k+1}}\langle O_k(s)x, O_k(s)y \rangle_{\mathcal H \otimes \mathbb R^{n^{\otimes k}}} \ ds
\end{equation}
are summable in operator norm. The limiting operator, given by $\mathscr{O} := \sum_{k=0}^{\infty} \mathscr{O}_k$, is called the \emph{observability gramian} $\mathscr{O}$ in $\mathcal L(X).$ 

Similarly for the reachability gramian, let $P_0(t_1):=T(t_1)^*$. Then, we define for $i \ge 1$ and $y \in X$
\begin{equation*}
P_{i}(t_1,..,t_{i+1})y:= \sum_{n_1,...,n_{i}=1}^n \left(\prod_{l=1}^{i}  \left(T(t_l)^*N_{n_l}^*\right)\right)T(t_{i+1})^*y \otimes \left(\widehat{e}_{n_1} \otimes ...\otimes \widehat{e}_{n_{i}}\right). 
\end{equation*}
The control operator $B \in \mathcal L(\mathbb R^n, X)$ shall be of the form $Bu=\sum_{i=1}^n \psi_i u_i$ for $\psi_i \in X.$ This implies that the operator $BB^*=\sum_{i=1}^n \langle \bullet, \psi_i \rangle \psi_i $ is a finite-rank operator. Then, we introduce operators $\mathscr{P}_k $ such that for any $x,y \in X$
\begin{equation}
\label{eq:P}
\langle x,\mathscr{P}_ky \rangle_X:=\int_{(0,\infty)^{k+1}} \left\langle P_k(s)x, \left(BB^* \otimes \operatorname{id}_{\mathbb R^{n^{\otimes k}}}\right) P_k(s)y \right\rangle_{X \otimes \mathbb R^{n^{\otimes k}}} \ ds.
\end{equation}
Let $M^2\Gamma^2(2\nu)^{-1} < 1,$ the \emph{reachability gramian} is then defined as $\mathscr P:= \sum_{k=0}^{\infty} \mathscr{P}_k \in \operatorname{TC}(X).$
The $\operatorname{TC}(X)$-convergence follows from the characterization \eqref{tracenorm} of the trace norm as for any orthonormal systems $(e_i),(f_i)$ of $X$ 
\begin{equation*}
\begin{split} 
&\sum_{i=1}^{\operatorname{dim}(X)} \left\lvert \langle f_i,\mathscr{P} e_i  \rangle_X \right\rvert  \\
&\le \left\lVert BB^* \right\rVert_{\operatorname{TC}(X)} \sum_{k=0}^{\infty}\int_{(0,\infty)^{k+1}} \sum_{n_1,...,n_{k}=1}^n  \left\lVert \prod_{l=1}^{k}  \left(T(t_l)^*N_{n_l}^*\right)T(t_{k+1})^* \right\rVert^2  \ dt < \infty.
\end{split}
 \end{equation*}

 \end{defi}
\begin{ass}
\label{generalass}
We assume that $M^2\Gamma^2(2\nu)^{-1}<1$ holds such that both $\mathscr{O}$ and $\mathscr{P}$ exist.
\end{ass}

As for systems on finite-dimensional spaces \cite[Theorems $3$ and $4$]{ZL} the gramians satisfy certain Lyapunov equations. However, those equations hold only in a weak sense if the generator of the semigroup $A$ is unbounded.
\begin{lemm}
For all $x_1,y_1 \in D(A)$ and all $x_2,y_2 \in D(A^*)$
\begin{equation}
\begin{split}
\label{Lyap1}
&\langle \mathscr{O}Ax_1,y_1 \rangle_X + \langle \mathscr{O}x_1,Ay_1 \rangle_X + \sum_{i=1}^n \langle \mathscr{O}N_ix_1,N_iy_1 \rangle_X + \left\langle Cx_1,Cy_1 \right\rangle_{\mathcal{H}} =0 \text{ and }\\
&\langle \mathscr{P}A^*x_2,y_2 \rangle_X + \langle \mathscr{P}x_2,A^*y_2 \rangle_X + \sum_{i=1}^n \langle \mathscr{P}N_i^*x_2,N_i^*y_2 \rangle_X + \langle BB^*x_2,y_2\rangle_X=0.
\end{split}
\end{equation}
\end{lemm}
\begin{proof} We restrict us to the proof of the first identity, since the proof of the second one is fully analogous.
Let $x \in D(A)$ then by \eqref{eq:O}
\begin{equation*} \begin{split}
&\langle \mathscr{O}_0Ax,x \rangle +\langle \mathscr{O}_0x,Ax \rangle_X + \left\lVert Cx \right\rVert_{\mathcal{H}}^2  \\
&=\int_{0}^{\infty}\left(\langle CT'(s)x,CT(s)x \rangle_{\mathcal{H}} + \langle CT(s)x, CT'(s)x \rangle_{\mathcal{H}}\right) \ ds +\left\lVert Cx \right\rVert_{\mathcal{H}}^2  \\
&=\int_{0}^{\infty} \frac{d}{ds} \left\lVert CT(s)x \right\rVert_{\mathcal{H}}^2 \ ds +\left\lVert Cx \right\rVert_{\mathcal{H}}^2=0. 
\end{split} \end{equation*}

Similarly, for $x \in D(A)$ and $k \ge 1$ by the fundamental theorem of calculus, the exponential decay of the semigroup at infinity, and the definition of the observability gramian
\begin{equation*} \begin{split}
&\langle \mathscr{O}_kAx,x \rangle_X + \langle \mathscr{O}_k x,Ax \rangle_X+\sum_{i=1}^n\langle \mathscr{O}_{k-1}N_ix,N_ix \rangle_X = \sum_{i=1}^n \langle \mathscr{O}_{k-1}N_ix,N_ix \rangle_X\\
&+  \sum_{i=1}^n\int_{(0,\infty)^{k}} \int_{(0,\infty)}\ \frac{d}{d\tau} \left\lVert O_{k-1}(s_1,....,s_k)(N_i T(\tau)x) \right\rVert^2_{\mathcal H \otimes \mathbb R^{n^{\otimes (k-1)}}} \ d\tau \ ds =0.
\end{split} \end{equation*}
Using uniform convergence of $\mathscr{O} = \sum_{k=0}^{\infty} \mathscr{O}_k$ it follows that
\begin{equation*}
\langle \mathscr{O} Ax,x \rangle_X + \langle \mathscr{O}x,Ax \rangle_X + \sum_{i=1}^n \langle \mathscr{O}N_ix,N_ix \rangle_X + \left\lVert Cx \right\rVert^2_{\mathcal{H}}=0. 
\end{equation*}
Finally, we may use the polarization identity to obtain \eqref{Lyap1}. 
\end{proof}

Analogously to the result for finite-dimensional systems in \cite[Theorem $3.1$]{BD}, we obtain the following eponymous properties for the gramians.  
\begin{lemm}
\label{homsys}
All elements $\varphi_0\in \operatorname{ker}(\mathscr{O})$ are unobservable in the homogeneous system, i.e. solutions to 
\begin{equation}
\label{homogeneousevoleq}
\varphi'(t)=A\varphi(t)+\sum_{i=1}^n N_i\varphi(t)u_i(t),\ \text{ for }t>0 
 \end{equation}
 with $\varphi(0)=\varphi_0  \in  \operatorname{ker}(\mathscr{O})$ satisfy $C\varphi(t)=0$ for all $t \ge 0.$
\end{lemm}
\begin{proof}
An element $x \in X$ is in $\operatorname{ker}(\mathscr O)$ if and only if $\left\langle \mathscr O_k x, x \right\rangle_X =0$ for all $k \in \mathbb N_0.$
We start by showing that $\operatorname{ker}(\mathscr O)$ is an invariant subspace of the semigroup $(T(t)).$
Let $x \in \operatorname{ker}(\mathscr O)$ then for all $t \ge 0$ and all $k$ by \eqref{eq:O} and the semigroup property 
\begin{equation*}
\begin{split}
0 \le \left\langle \mathscr O_k T(t)x, T(t)x \right\rangle_X &=\int_{(0,\infty)^{k+1}}\left\lVert O_k(s)T(t)x \right\rVert^2_{\mathcal H \otimes \mathbb R^{n^{\otimes k}}} \ ds\\
&= \int_{(0,\infty)^{k}} \sum_{i=1}^n \int_{0}^{\infty}\left\lVert O_{k-1}(s)N_i T(s_{k+1}+t)x\right\rVert^2_{\mathcal H \otimes \mathbb R^{n^{\otimes {k-1}}}} \ ds_{k+1} \ ds \\
&= \int_{(0,\infty)^{k}} \int_{t}^{\infty}\langle O_k(s,\tau)x, O_k(s,\tau)x \rangle_{\mathcal H \otimes \mathbb R^{n^{\otimes k}}} \ d\tau \ ds \\
&\le \int_{(0,\infty)^{k+1}} \langle O_k(s)x, O_k(s)x \rangle_{\mathcal H \otimes \mathbb R^{n^{\otimes k}}}  \ ds = \left\langle \mathscr O_k x, x \right\rangle_X = 0
\end{split}
\end{equation*} 
where we used the semigroup property of $(T(t))$, substituted $\tau=s_{k+1}+t,$ and extended the integration domain to get the final inequality.
Thus, $(T(t))$ restricts to a $C_0$-semigroup on the closed subspace $\operatorname{ker}(\mathscr O)$ and the generator of $A$ is the part of $A$ in $\operatorname{ker}(\mathscr O)$ \cite[Chapter II 2.3]{EN}. In particular, $D(A) \cap \operatorname{ker}(\mathscr O)$ is dense in $\operatorname{ker}(\mathscr O).$
Let $x \in \operatorname{ker}(\mathscr{O}) \cap D(A),$ then positivity of $\mathscr{O}$ implies by the first Lyapunov equation \eqref{Lyap1} with $x_1=y_1=x$
that $N_ix \in \operatorname{ker}(\mathscr{O})$ and $x \in \operatorname{ker}(C)$. Thus, a density argument shows $N_i \left(\operatorname{ker}(\mathscr{O}) \right) \subset \operatorname{ker}(\mathscr{O})$ and $\operatorname{ker}(\mathscr{O})\subset \operatorname{ker}(C).$

This shows, by \cite[Proposition 5.3]{LiYo}, that \eqref{homogeneousevoleq} is well-posed on $\operatorname{ker}(\mathscr O)$, i.e. for initial data in $\operatorname{ker}(\mathscr O)$ the solution to \eqref{homogeneousevoleq} stays in $\operatorname{ker}(\mathscr{O})$.
From the inclusion $\operatorname{ker}(\mathscr{O}) \subset \operatorname{ker}(C)$, we then obtain $C \varphi(t)=0.$
\end{proof}

\begin{lemm}
The closure of the range of the reachability gramian $\mathscr{P}$ is an invariant subspace of the flow of \eqref{evoleq}, i.e. for $\varphi_0 \in \overline{\operatorname{ran}}(P)$ it follows that $\varphi(t) \in \overline{\operatorname{ran}}(P)$ for all times $t\ge 0.$
\end{lemm}
\begin{proof}
Analogous to Lemma \ref{homsys}.
\end{proof}

\section{Hankel operators on Fock spaces}\label{sec:Hankel}
To decompose the observability gramian as $\mathscr{O}= W^*  W$ and the reachability gramian as $\mathscr{P} =  R  R^* $, we start by defining the observability and reachability maps.
\begin{defi}
\label{defWk}
For $k \in \mathbb{N}_0$ let $ W_k\in \mathcal L\left( X, F^n_{k+1}\left(\mathcal{H}  \right)\right)$ be the operators that map $X \ni x \mapsto O_k(\bullet)x.$
Their operator norms can be bounded by $\left\lVert  W_k \right\rVert = \mathcal O\left( \left(M \Gamma ( 2 \nu )^{-1/2} \right)^k\right).$ \newline
A straightforward computation shows that the adjoint operators $ W_k^* \in \mathcal L\left(F_{k+1}^n\left(\mathcal{H} \right),X\right)$ read
\begin{equation*}
 W_k^*f:=\int_{(0,\infty)^{k+1}} O_k^*(s)f(s) \ ds.
\end{equation*}
Then, we can define, by Assumption \ref{generalass}, the \emph{observability map} $W \in \mathcal L\left(X, F^n\left(\mathcal{H} \right)\right)$ as $W(x):=\left( W_k(x)\right)_{k \in \mathbb{N}_0}.$
An explicit calculation shows that  $ W^*$ is given for $(f_k)_{k}\in F^n\left(\mathcal{H} \right)$ by
\[ W^*((f_k)_{k})= \sum_{k=0}^{\infty} W_k^*f_k.\]

Similarly to the decomposition of the observability gramian, we introduce a decomposition of the reachability gramian $\mathscr{P} =  R  R^*$.
Let  
\begin{equation*}
 R_k \in \operatorname{HS}\left(F_{k+1}^n\left(\mathbb{R}^{n}\right),X\right) \text{ be given by } R_kf:=\int_{(0,\infty)^{k+1}} P_k(s)^*(B\otimes \operatorname{id}_{\mathbb R^{n^{\otimes k}}}) f(s) \ ds.
\end{equation*}
The adjoint operators of the $R_k$ are the operators 
\begin{equation*}
 R_k^* \in \operatorname{HS}\left(X,F_{k+1}^n\left(\mathbb{R}^{n} \right)\right) \text{ with } R_k^*x:= \left(B^*\otimes \operatorname{id}_{\mathbb R^{n^{\otimes k}}}\right) P_k(\bullet)x. 
\end{equation*}
If the gramians exist, then the \emph{reachability map} is defined as 
\[ R \in \operatorname{HS}\left(  F^n\left(\mathcal{H}\right),X\right) \text{ such that } (f_k)_{k \in \mathbb{N}_0} \mapsto \sum_{k=0}^{\infty}  R_k f_k.\]
Its adjoint is given by $ R^* \in \operatorname{HS}\left(X, F^n(\mathbb{R}^n)\right), \ 
X \ni x \mapsto \left(  R_k^*(x) \right)_{k \in \mathbb{N}_0}.$
\end{defi}
To see that $ R_k$ is a Hilbert-Schmidt operator we take an ONB $(e_i)$ of $F_{k+1}^n\left(\mathbb{R}^{n}\right)$, such that the $e_i$ are tensor products of an ONB of $L^2((0,\infty),\mathbb R)$ and standard unit vectors of $\mathbb R^n$, and an arbitrary ONB $(f_j)$ of $X$
\begin{equation}
\begin{split}
\label{eq:HSP}
&\left\lVert R_k \right\rVert^2_{\operatorname{HS}\left(F_{k+1}^n\left(\mathbb{R}^{n}\right),X\right)}
=\sum_{j=1}^{\operatorname{dim}(X)}\sum_{i=1}^{\infty} \left\lvert \left\langle f_j, R_k e_i \right\rangle_X \right\rvert^2=\sum_{j=1}^{\operatorname{dim}(X)}\sum_{i=1}^{\infty} \left\lvert \left\langle R_k^*f_j, e_i \right\rangle_{F_{k+1}^n\left(\mathbb{R}^{n}\right)} \right\rvert^2 \\
&= \sum_{j=1}^{\operatorname{dim}(X)} \sum_{i=1}^n \sum_{n_1,...,n_k=1}^n\int_{(0,\infty)^{k+1}} \left\lvert \left\langle f_j, P_k(s)^*(\psi_i\otimes \widehat{e}_{n_1} \otimes ...\otimes \widehat{e}_{n_k}) \right\rangle_{X} \right\rvert^2 \ ds  \\
&=  \sum_{i=1}^n \sum_{n_1,...,n_k=1}^n\int_{(0,\infty)^{k+1}} \left\lVert P_k(s)^*(\psi_i\otimes \widehat{e}_{n_1} \otimes ...\otimes \widehat{e}_{n_k}) \right\rVert_{X}^2  \ ds = \mathcal O \left( \left(M^{2} \Gamma^{2}( 2 \nu )^{-1} \right)^k \right).  
\end{split}
 \end{equation}
One can then check that the maps $W$ and $P$ indeed decompose the gramians as $\mathscr{O} =  W^* W$ and $\mathscr{P} =  R  R^*.$
We now introduce the main object of our analysis:
\begin{defi}
\label{HSop}
The \emph{Hankel operator} is the Hilbert-Schmidt operator $H:= W  R \in \operatorname{HS}\left(F^n(\mathbb{R}^n),F^n(\mathcal{H})\right).$
\end{defi}
Since any compact operator acting between Hilbert spaces possesses a singular value decomposition, we conclude that:
\begin{corr}
\label{regularitytheo} 
There are $(e_k)_{k \in \mathbb{N}} \subset F^n(\mathbb{R}^n) $ and $(f_k)_{k \in \mathbb{N}} \subset F^n(\mathcal H)$-orthonormal systems as well as singular values $(\sigma_k)_{k \in \mathbb{N}} \in \ell^2(\mathbb{N})$ such that 
\begin{equation}
\label{Henkelop}
H = \sum_{k=1}^{\infty} \sigma_k\langle \bullet,e_k \rangle_{F^n(\mathbb{R}^n)} f_k, \quad
He_k= \sigma_k f_k, \text{ and } 
H^*f_k = \sigma_k e_k.
\end{equation}
\end{corr}

We now state a sufficient condition under which $H$ becomes a trace class operator such that $(\sigma_k)_{k \in \mathbb{N}} \in \ell^1(\mathbb{N}).$
\begin{lemm}
\label{traceclasslemma}
If $\mathcal H\simeq\mathbb{R}^m$ for any $m \in \mathbb N$ then $ W$ is a Hilbert-Schmidt operator just like $R$. Consequently, $H =  W  R \text{ and } \mathscr{O}=  W^* W $
are both of trace class.
\end{lemm}
\begin{proof} 
Since for any $i \in \left\{1,..,m\right\}$ and $i_1,...,i_k \in \left\{1,..,n \right\}$ the operator 
\[X \ni x \mapsto \left\langle \widehat{e}_i \otimes \widehat{e}_{i_1} \otimes...\otimes \widehat{e}_{i_k},W_kx \right\rangle_{\mathbb{R}^m \otimes \mathbb R^{n^{\otimes k}}}=:Q_{i,i_1,...,i_k}(x) \] is a \emph{Carleman operator}, we can apply \cite[Theorem $6.12$(iii)]{Wei} that characterizes Carleman operators of Hilbert-Schmidt type. The statement of the Lemma follows from the summability of
\begin{equation*}
\begin{split}
&\left\lVert W_k \right\rVert_{\operatorname{HS}}^2
= \sum_{i=1}^m  \sum_{i_1,..,i_k=1}^n \left\lVert  Q_{i,i_1,...,i_k}  \right\rVert_{\operatorname{HS}(X,L^2((0,\infty)^{k+1},\mathbb R))}^2 \\
&\le \sum_{i=1}^m  \sum_{i_1,..,i_k=1}^n  \int_{(0,\infty)^{k+1}} \left\lVert O_k(t)^*\left(\widehat{e}_i \otimes \widehat{e}_{i_1} \otimes...\otimes \widehat{e}_{i_k}\right) \right\rVert^2_{X} \ dt =\mathcal O\left(\left(M^{2} \Gamma^{2} (2 \nu )^{-1}\right)^{k}\right).
\end{split}
\end{equation*}
\end{proof}
In the rest of this section, we discuss immediate applications of our preceding construction. We start by introducing the truncated gramians.
\begin{defi}
\label{trungram}
The \emph{k-th order truncation} of the gramians are the first $k$ summands of the gramians, i.e. $\mathscr{O}^{(k)}:=\sum_{i=0}^{k-1}\mathscr{O}_i$ and $\mathscr{P}^{(k)}:=\sum_{i=0}^{k-1}\mathscr{P}_i.$
The associated $k$-th order truncated Hankel operator is $H^{(k)}f:=( W_i \sum_{j=0}^{k-1}  R_jf_j)_{i \in \{0,...,k-1\}}$.
\end{defi}
The proof of Proposition \ref{singularvalueconv} follows then from our preliminary work very easily:
\begin{proof}[Proof of Proposition \ref{singularvalueconv}]
From \cite[Corollary $2.3$]{Krein} it follows that for any $i \in \mathbb{N}$ the difference of singular values can be bounded as $\left\lvert \sigma_{i}-\sigma^k_{i} \right\rvert \le \left\lVert H-H^{(k)} \right\rVert \le \left\lVert H-H^{(k)} \right\rVert_{\operatorname{HS}}$ and by the inverse triangle inequality $\left\lvert \left\lVert \sigma \right\rVert_{\ell^2}-\left\lVert \sigma^{k} \right\rVert_{\ell^2} \right\rvert \le\left\lVert H-H^{(k)} \right\rVert_{\operatorname{HS}}$. Thus, it suffices to bound by \eqref{eq:HSP} and Definition \ref{defWk}
\begin{equation*}
\begin{split}
\left\lVert H-H^{(k)} \right\rVert_{\operatorname{HS}}^2 = \sum_{(i,j) \in \mathbb{N}_0^2 \backslash \{0,...,k-1\}^2} \left\lVert H_{ij} \right\rVert^2_{\operatorname{HS}} &=  \sum_{(i,j) \in \mathbb{N}_0^2 \backslash \{0,...,k-1\}^2} \left\lVert  W_i \right\rVert^2 \left\lVert  R_j \right\rVert^2_{\operatorname{HS}} \\
&= \mathcal O \left(\left(M^{2} \Gamma^{2} ( 2 \nu )^{-1} \right)^{2k} \right).
\end{split}
\end{equation*}
\end{proof}
Next, we state the proof of Proposition \ref{finapprox} on the approximation by subsystems. The Hankel operator for the subsystem on $V_i$ is then just given by $H_{V_i}:=WR_{V_i}$ where
\[R_{V_i}(f):=\sum_{k=0}^{\infty} \int_{(0,\infty)^{k+1}} P_k(s)^*(P_{V_i}B\otimes \operatorname{id}_{\mathbb R^{n^{\otimes k}}}) f_k(s) \ ds\]
with $P_{V_i}$ being the orthogonal projection onto $V_i.$
\begin{proof}[Proof of Proposition \ref{finapprox}]
Using elementary estimates 
\[\left\lVert H-H_{V_i} \right\rVert_{\operatorname{TC}} \le  \left\lVert W \right\rVert_{\operatorname{HS}}\left\lVert R-R_{V_i} \right\rVert_{\operatorname{HS}}\text{ and }\left\lVert H-H_{V_i} \right\rVert_{\operatorname{HS}} \le  \left\lVert W \right\rVert\left\lVert R-R_{V_i} \right\rVert_{\operatorname{HS}},\]
it suffices to show $\operatorname{HS}$-convergence of $R_{V_i}$ to $R.$ This is done along the lines of \eqref{eq:HSP}.
\end{proof}

\subsection{Convergence of singular vectors}
The convergence of singular values has already been addressed in Proposition \ref{singularvalueconv}.
For the convergence of singular vectors, we now assume that there is a family of compact operators $H(m) \in \mathcal L \left(F^n\left(\mathbb R^n \right),F^n\left(\mathcal H \right) \right)$ converging in operator norm to $H$. By compactness, every operator $H(m)$ has a singular value decomposition $H(m)= \sum_{k=1}^{\infty} \sigma_k(m) \langle \bullet,e_{k}(m) \rangle f_{k}(m).$

\begin{ass}
Without loss of generality let the singular values be ordered as $\sigma_1(m) \ge \sigma_2(m) \ge..$ . 
Furthermore, for the rest of this section, all singular values of $H$ are assumed to be non-zero and non-degenerate, i.e. all eigenspaces of $HH^*$ and $H^*H$ are one-dimensional. \end{ass}
\begin{lemm}
\label{SVC}
Let the family of compact operators $(H(m))$ converge to the Hankel operator $H$ in operator norm, then the singular vectors convergence in norm as well. 
\end{lemm}
\begin{proof}[Proof of Lemma \ref{SVC}]
We formulate the proof only for singular vectors $(e_j)$ since the arguments for $(f_j)$ are analogous. We start by writing $e_j=r(m)e_j(m)+x_j(m)$ where $\langle e_j(m),x_j(m) \rangle=0.$ Then, the arguments stated in the proof of  \cite[Appendix 2]{CGP} show that for $m$ sufficiently large (the denominator is well-defined as the singular values are non-degenerate) 
\begin{equation*} \begin{split}
&\left\lVert x_j(m) \right\rVert_{F^n(\mathbb{R}^n)}^2 \le \frac{\sigma_j^2-\left(\sigma_j-2\left\lVert H_j-H_j(m) \right\rVert_{\mathcal L\left(F^n\left(\mathbb R^n \right),F^n\left(\mathcal H \right) \right)}\right)^2}{\sigma_j^2-\sigma_{j+1}^2}  \xrightarrow[m \rightarrow \infty]{} 0
\end{split} \end{equation*}
where $H_j:=H-\sum_{k=0}^{j-1} \sigma_k \langle\bullet,e_k \rangle f_k$ and $H_j(m):=H(m)-\sum_{k=0}^{j-1} \sigma_k(m) \langle\bullet,e_k(m) \rangle f_k(m). $
\end{proof}

\section{Global error estimates}\label{sec:errorBound}
We start by defining a \emph{control tensor} $U_k(s) \in \mathcal L\left(\mathcal H \otimes \mathbb R^{n^{\otimes k}}, \mathcal H\right) $ 
\[ U_k(s):=\sum_{i_1,..,i_k=1}^n u_{i_1}(s_1)\cdot...\cdot u_{i_k}(s_k)\operatorname{id}_{\mathcal H} \otimes \  \left\langle \widehat{e}_{i_1}\otimes ...\otimes \widehat{e}_{i_k}, \bullet \right\rangle. \]
Using sets $\Delta_k(t):=\{(s_1,...,s_k) \in \mathbb{R}^k; 0 \le s_k \le ...\le s_1 \le t\}$, we can decompose the output map $(0,\infty) \ni t \mapsto C\varphi(t)$ with $\varphi$ as in \eqref{mild solution} for controls $\left\lVert u \right\rVert_{L^2((0,\infty),(\mathbb R^n, \left\lVert \bullet \right\rVert_{\infty}))}  < \frac{\sqrt{2 \nu}}{M \Xi}$ and $\Xi:=\sum_{i=1}^n \left\lVert N_i \right\rVert$ according to Lemma \ref{Voltconv} into two terms
$C\varphi(t) =  K_1(t)+K_2(t)$ such that
\begin{equation}
\begin{split}
\label{eq:G}
K_1(t)&:= \sum_{k=1}^{\infty} \int_{\Delta_k(t)} U_k(s) \left(O_{k}(t-s_1,.,s_{k-1}-s_k,s_k)\varphi_0\right)  \ ds + CT(t) \varphi_0 \text{ and } \\
K_2(t)&:= \sum_{k=1}^{\infty} \int_{\Delta_k(t)}  U_k(s) \left(\sum_{i=1}^nO_{k-1}(t-s_1,s_1-s_2,...,s_{k-1}-s_k)\psi_i \otimes \widehat{e}_i \right) \  ds.
\end{split}
\end{equation}
The first term $K_1$ is determined by the initial state $\varphi_0$ of the evolution problem \eqref{evoleq}. If this state is zero, the term $K_1$ vanishes. 
The term $K_2$ on the other hand captures the \emph{intrinsic} dynamics of equation \eqref{evoleq}. 
A technical object that links the dynamics of the evolution equation with the operators from the balancing method are the Volterra kernels we study next.
\begin{defi}
\label{def:Voltker}
The \emph{Volterra kernels} associated with \eqref{evoleq} are the functions
\begin{equation*}
\begin{split}
&h_{k,j} \in L^2\left((0,\infty)^{k+j+1},\operatorname{HS}\left(\mathbb R^{n^{\otimes (j+1)}}, \mathcal H \otimes \mathbb R^{n^{\otimes k}} \right)\right)\\
&h_{k,j}(\sigma_0,...,\sigma_{k}+\sigma_{k+1},..,\sigma_{k+j+1}):=O_{k}(\sigma_0,...,\sigma_k)P_{j}^*(\sigma_{k+j+1},...,\sigma_{k+1})(B\otimes \operatorname{id}_{\mathbb R^{n^{\otimes {j}}}}).
\end{split}
\end{equation*}
\end{defi}
The Volterra kernels satisfy an invariance property for all $p,q,k,j \in \mathbb N_0$ such that $p+q=k+j:$
\begin{equation}
\label{eq:invp}
 \left\lVert   h_{k,j} \right\rVert_{L^{1}_{k+1}L^{2}_{k+j}\left(\operatorname{HS}\left(\mathbb R^{n^{\otimes (j+1)}},\mathcal H \otimes \mathbb R^{n^{\otimes k}}\right)\right)}  =  \left\lVert   h_{p,q} \right\rVert_{L^{1}_{k+1}L^{2}_{k+j}\left(\operatorname{HS}\left(\mathbb R^{n^{\otimes (q+1)}},\mathcal H \otimes \mathbb R^{n^{\otimes p}}\right)\right)}. 
 \end{equation}
The Volterra kernels are also the integral kernels of the components of the Hankel operator
\begin{equation*} 
\begin{split} 
\left( W_k R_j f \right)(s_0,...,s_k)  = \int_{(0,\infty)^{j+1}} h_{k,j} (s_0,...,s_k+t_{1},...,t_{j+1}) f(t) \ dt.
\end{split}
\end{equation*}
\begin{rem}
In particular the kernels $h_{k,0}$ appear in the definition of the $\mathscr H^2$-system norm introduced in \cite[Eq. 15]{ZL}
\begin{equation*}
\begin{split}
\left\lVert \Sigma \right\rVert^2_{\mathscr{H}^2}&:= \sum_{k=0}^{\infty} \left\lVert h_{k,0} \right\rVert^2_{L^2\left((0,\infty)^{k+1},\operatorname{HS}\left(\mathbb R^{n},\mathcal H \otimes \mathbb R^{n^{\otimes k}}\right)\right)} \\
&=\sum_{k=0}^{\infty}  \int_{(0,\infty)^{k+1}} \sum_{n_1,...,n_{k}=1}^n \left\lVert CT(t_1)  \prod_{l=2}^{k+1} \left(N_{n_{l-1}}T(t_l) \right)B  \right\rVert_{\operatorname{HS}(\mathbb R^n, \mathcal H)}^2 \ dt
\end{split}
\end{equation*}
for which robust numerical algorithms with strong $\mathscr H^2$-error performance are available \cite{BB11}.

This system norm can also be expressed directly in terms of the gramians 
\[\left\lVert \Sigma \right\rVert^2_{\mathscr{H}^2} = \tr\left(BB^*\mathscr O\right)= \tr\left(C^*C\mathscr P\right)\]
which is well-defined as $B^*B$ and $\mathscr P$ are both trace class operators.
\end{rem}
In \cite{BD2} the $k$-th order \emph{transfer function} $G_k$ has been introduced as the $k+1$-variable Laplace transform of the Volterra kernel $h_{k,0}$
\begin{equation*}
G_{k}(s):=\int_{(0,\infty)^{k+1}} h_{k,0}(t)e^{- \langle s, t \rangle} \ dt.  
\end{equation*}
Using mixed Hardy norms as defined in \eqref{eq: mix}, the Paley-Wiener theorem implies the following estimate for $i \in \left\{1,..,k+1 \right\}$
\begin{equation}
\begin{split}
\label{Fourierestimate}
&\left\lVert G_{k} \right\rVert_{\mathscr H^{\infty}_{i} \mathscr H^{2}_{k}\left(\operatorname{HS}\left(\mathbb R^n,\mathcal H \otimes \mathbb R^{n^{\otimes k}}\right)\right)} \\
& \le \int_{0}^{\infty} \left\lVert \int_{(0,\infty)^{k}}h_{k,0}^{(i)}(s,\sigma)e^{-\langle \bullet, s \rangle} \ ds \right\rVert_{\mathscr{H}^2\left((0,\infty)^{k},\operatorname{HS}\left(\mathbb R^n,\mathcal H \otimes \mathbb R^{n^{\otimes k}}\right)\right)} \ d\sigma \\
&=\left\lVert h_{k,0}^{(i)} \right\rVert_{L^{1}_{i}L^{2}_{k}\left(\operatorname{HS}\left(\mathbb R^n,\mathcal H \otimes \mathbb R^{n^{\otimes k}}\right)\right)}.
\end{split}
\end{equation}

For two systems $\Sigma$ and $\widetilde \Sigma$ satisfying Assumption \ref{generalass} with the same number of controls and the same output space $\mathcal H$, we then define the \emph{difference Volterra kernel} and the \emph{difference Hankel operator} $\Delta(h):=h-\widetilde h$ and $\Delta(H):= H -\widetilde H =\left( W_i R_j - \widetilde{W_i} \widetilde{ R_j} \right)_{ij}.$

In the following Lemma we derive a bound on the mixed $L^1$-$L^2$ norm of the Volterra kernels:

\begin{lemm}
\label{Hankelestimate}
Consider two systems satisfying Assumption \ref{generalass} with the same number of controls and the same output space $\mathcal H \simeq \mathbb R^m$ such that $H$ is trace class (Lemma \ref{traceclasslemma}). Then the Volterra kernels $h_{k,j}$ satisfy
\begin{equation*}
 \left\lVert  \Delta( h_{k,j}) \right\rVert_{L^{1}_{k+1}L^{2}_{k+j}\left(\operatorname{HS}\left(\mathbb R^{n^{\otimes (j+1)}},\mathbb R^m \otimes \mathbb R^{n^{\otimes k}}\right)\right)} \le 2 \left\lVert \Delta(  W_k  R_j) \right\rVert_{\operatorname{TC}\left(F^n_{j+1}(\mathbb R^n),F^n_{k+1}\left(\mathbb R^m\right)\right)}.
\end{equation*} 
\end{lemm}
\begin{proof}
Given the difference Volterra kernel $\Delta(h_{k,j})$ associated with $\Delta( W_k R_j).$ 

For every $z \in \mathbb{N}_0$ and $\alpha>0$ fixed, we introduce the family of sesquilinear forms $(L_{z,\alpha})$
\begin{equation*} 
\begin{split}
\label{form}
&L_{z,\alpha}:  F_k^1\left(\mathbb{R}^{m}\otimes \mathbb R^{n^{\otimes k}}\right) \oplus F_j^1\left(\mathbb{R}^{n^{\otimes (j+1)}}\right) \rightarrow \mathbb{R}  \\
& (f,g) \mapsto \int_{(0,\infty)^{k+j}} \left\langle f(s_1,..,s_{k}),\Delta\left(h^{(k+1)}_{k,j}(s,2z\alpha)\right) g(s_{k+1},..,s_{k+j}) \right\rangle_{\mathbb{R}^{m}\otimes \mathbb R^{n^{\otimes k}}} \ ds.  
\end{split}  
\end{equation*}
Since $\Delta\left(h^{(k+1)}_{k,j}(\bullet,2z\alpha)\right) \in F_k^1\left(\mathbb{R}^{m}\otimes \mathbb R^{n^{\otimes k}}\right) \otimes F_j^1\left(\mathbb{R}^{n^{\otimes (j+1)}}\right)=: Z$ we can define a Hilbert-Schmidt operator \footnote{ For separable Hilbert spaces $H_1$ and $H_2$ there is the isometric isomorphism $H_1 \otimes H_2 \equiv \operatorname{HS}(H_1^*,H_2).$} of unit Hilbert-Schmidt norm given by $Q : F_j^1\left(\mathbb{R}^{n^{\otimes (j+1)}}\right) \rightarrow F_k^1\left(\mathbb{R}^{m}\otimes \mathbb R^{n^{\otimes k}}\right)$
\begin{equation*}
\begin{split}
&(Q\varphi)(s):=\int_{(0,\infty)^{j}} \frac{\Delta \left(h^{(k+1)}_{k,j}((s,t),2z\alpha)\right)}{\left\lVert \Delta\left(h^{(k+1)}_{k,j}(\bullet,2z\alpha)\right)\right\rVert_{Z}} \varphi (t) \ dt.
\end{split}
\end{equation*}
Doing a singular value decomposition of $Q$ yields orthonormal systems $f_{z,i} \in F_k^1\left(\mathbb{R}^{m}\otimes \mathbb R^{n^{\otimes k}}\right)$, $g_{z,i} \in F_j^1\left(\mathbb{R}^{n^{\otimes (j+1)}}\right)$, parametrized by $i \in \mathbb N,$ and singular values $\sigma_{z,i} \in [0,1]$ such that for any $\delta>0$ given there is $N(\delta)$ large enough with
\[ \left\lVert \frac{\Delta\left(h^{(k+1)}_{k,j}(\bullet,2z\alpha)\right)}{\left\lVert \Delta\left(h^{(k+1)}_{k,j}(\bullet,2z\alpha)\right)\right\rVert_{Z}} - \sum_{i=1}^{N(\delta)} \sigma_{z,i} (f_{z,i} \otimes g_{z,i}  )\right\rVert_{Z}<\delta. \]

Let $\varepsilon>0,$ then for $M$ sufficiently large $\int_{M}^{\infty} \left\lVert \Delta\left(h^{(k+1)}_{k,j}(\bullet,v)\right)\right\rVert_{Z} \ dv < \varepsilon.$
Thus, for $z \in \mathbb{N}_0$ there are $f_{z,i}\in F_k^1\left(\mathbb{R}^{m}\otimes \mathbb R^{n^{\otimes k}}\right)$ and $g_{z,i}  \in F_j^1\left(\mathbb{R}^{n^{\otimes (j+1)}}\right)$ orthonormalized, $\sigma_{z,i} \in [0,1]$, and $N_z \in \mathbb N$ such that 
\begin{equation}
\begin{split}
\label{functionalbound}
&\left\lvert \left\langle  \frac{\Delta\left(h^{(k+1)}_{k,j}(\bullet,2z\alpha)\right)}{\left\lVert \Delta\left(h^{(k+1)}_{k,j}(\bullet,2z\alpha)\right)\right\rVert_{Z}}-\sum_{i=1}^{N_z} \sigma_{z,i} (f_{z,i} \otimes g_{z,i} ), \Delta\left(h^{(k+1)}_{k,j}(\bullet,2z\alpha)\right)\right\rangle_{Z} \right\rvert \\
&=\left\lvert\left\lVert \Delta\left(h^{(k+1)}_{k,j}(\bullet,2z\alpha)\right)\right\rVert_{Z} -\sum_{i=1}^{N_z} \sigma_{z,i} L_{z,\alpha}(f_{z,i},g_{z,i}) \right\rvert <\frac{ \varepsilon}{M}. 
\end{split}
\end{equation}

Then, $s_{z,i}(r,u):= \frac{1}{\sqrt{\alpha}} \indic_{[z\alpha,(z+1)\alpha)}(r)g_{z,i}(u)$ and $t_{z,i}(r,u):= \frac{1}{\sqrt{\alpha}} \indic_{[z\alpha,(z+1)\alpha)}(r) f_{z,i}(u)$ form orthonormal systems parametrized by $z$ and $i$ in spaces $F^n_{j+1}(\mathbb{R}^n)$ and $F_{k+1}^n\left(\mathbb{R}^{m}\right)$ respectively, such that using the auxiliary quantities
\begin{equation*}
\begin{split}
&I:=(z \alpha,(z+1)\alpha)^2 \times (0,\infty)^{k+j}, \ J:=(2z \alpha,2(z+1)\alpha) \times (0,\infty)^{k+j}, \text{ and }\\
&\lambda(v):=\min\left\{v-2z\alpha,2(z+1)\alpha-v\right\}
\end{split}
\end{equation*}
it follows that
\begin{equation} 
\begin{split} 
\label{eq:char}
&\langle t_{z,i}, \Delta( W_k  R_j) s_{z,i} \rangle_{F_{k+1}^n (\mathbb R^m)} \\
&=\frac{1}{\alpha}\int_{I} \left\langle f_{z,i}(s_1,.,s_{k}),\Delta\left(h^{(k+1)}_{k,j}\right)(s,r+t) g_{z,i}(s_{k+1},.,s_{k+j})\right\rangle_{\mathbb{R}^{m}\otimes \mathbb R^{n^{\otimes k}}} \ dr \ dt \ ds \\
&=\frac{1}{2\alpha} \int_{J}  \int_{-\lambda(v)}^{\lambda(v)}  \left\langle f_{z,i}(s_1,.,s_{k}),\Delta\left(h^{(k+1)}_{k,j}\right)(s,v) g_{z,i}(s_{k+1},.,s_{k+j})\right\rangle_{\mathbb{R}^{m}\otimes \mathbb R^{n^{\otimes k}}}  \ dw \ dv \ ds\\
&=\frac{1}{\alpha} \int_{J} \lambda(v) \left\langle f_{z,i}(s_1,.,s_{k}),\Delta\left(h^{(k+1)}_{k,j}\right)(s,v) g_{z,i}(s_{k+1},.,s_{k+j})\right\rangle_{\mathbb{R}^{m}\otimes \mathbb R^{n^{\otimes k}}}  \ dv \ ds
\end{split} 
\end{equation}
where we made the change of variables $v:=r+t$ and $w:=r-t.$ 
For $\alpha$ small enough and $v_1, v_2 \in [0,M+1]$ we have by strong continuity of translations
\begin{equation}
\label{eq: uniformc}
 \left\lVert \Delta\left(h^{(k+1)}_{k,j}(\bullet,v_1)\right)-\Delta\left(h^{(k+1)}_{k,j}(\bullet,v_2)\right) \right\rVert_{Z} < \frac{\varepsilon}{M} \text{ if } \left\lvert v_1-v_2 \right\rvert < 2 \alpha.
\end{equation}
Hence, using the above uniform continuity as well as \eqref{functionalbound} and \eqref{eq:char} 
\begin{equation*} 
\begin{split}
&\left\lvert \sum_{i=1}^{N_z} \sigma_{z,i} \langle   t_{z,i}, \Delta( W_k  R_j) s_{z,i} \rangle_{F_{k+1}^n (\mathbb R^m)}    - \alpha  \left\lVert \Delta\left(h^{(k+1)}_{k,j}(\bullet,2z\alpha)\right)\right\rVert_Z  \right\rvert \\
&\le \frac{1}{\alpha} \int_{2z\alpha}^{2(z+1)\alpha} \lambda(v) \Bigg(\left\lvert \sum_{i=1}^{N_z} \sigma_{z,i} \left\langle  f_{z,i} \otimes g_{z,i}, \Delta\left(h^{(k+1)}_{k,j}(\bullet,v)\right)-\Delta\left(h^{(k+1)}_{k,j}(\bullet,2z\alpha)\right) \right\rangle_Z \right\rvert   \\
&\qquad \qquad \qquad \qquad  +\left\lvert  \sum_{i=1}^{N_z} \sigma_{z,i}  L_{z,\alpha}(f_{z,i},g_{z,i}) -\left\lVert\Delta\left(h^{(k+1)}_{k,j}(\bullet,2z\alpha)\right)\right\rVert_Z \right\rvert \Bigg) \ dv  \lesssim \frac{\alpha \varepsilon}{M}.
\end{split}
\end{equation*}

This implies immediately by uniform continuity \eqref{eq: uniformc}
\[\left\lvert \sum_{i=1}^{N_z} \sigma_{z,i}  \langle t_{z,i}, \Delta( W_k  R_j) s_{z,i} \rangle_{F_{k+1}^n (\mathbb R^m)}  - \frac{1}{2} \int_{2z\alpha}^{2(z+1)\alpha} \left\lVert \Delta\left(h^{(k+1)}_{k,j}(\bullet,v)\right)\right\rVert_Z dv \right\rvert \lesssim \frac{\alpha \varepsilon}{M}.
\]
Summing over $z$ up to $\left\lfloor{\frac{M}{2\alpha}} \right \rfloor $ implies by the choice of $M$ that
\begin{equation*} 
\begin{split}
&\left\lvert \sum_{z=0}^{\left \lfloor{\frac{M}{2\alpha}}\right \rfloor } \sum_{i=1}^{N_z} \sigma_{z,i}  \langle t_{z,i}, \Delta( W_k  R_j) s_{z,i} \rangle_{F_{k+1}^n (\mathbb R^m)} - \frac{1}{2}  \left\lVert \Delta\left(h^{(k+1)}_{k,j}\right) \right\rVert_{L^{1}_{k+1}L^{2}_{k+j}\left(\operatorname{HS}\right)} \right\rvert \lesssim \varepsilon.
\end{split}
\end{equation*}
The Lemma follows then from the characterization of the trace norm stated in \eqref{tracenorm}.
\end{proof}
The preceding Lemma implies bounds on the difference of the dynamics for two systems $\Sigma$ and $\widetilde{\Sigma}$ satisfying Assumption \ref{generalass}. 
Before explaining this in more detail, we recall the notation $\Delta(X):=X-\widetilde{X}$ used in the introduction where $X$ is some observable of system $\Sigma$ and $\widetilde{X}$ its pendant in system $\widetilde{\Sigma}$.

In particular, Lemma \ref{Hankelestimate} immediately gives the statement of Theorem \ref{singestim}.
\begin{proof}[Proof of Theorem \ref{singestim}]
The Hankel operator is an infinite matrix with operator-valued entries $H_{ij}= W_iR_j.$ 
Using the invariance property \eqref{eq:invp}, we can combine Lemma \ref{Hankelestimate} with estimate \eqref{Fourierestimate}, relating the transfer functions to the Volterra kernels, to obtain from the definition of the trace norm \eqref{tracenorm} that
\begin{equation*}
\begin{split}
&\sum_{k=1}^{\infty} \left\lVert \Delta(G_{2k-1}) \right\rVert_{\mathscr H^{\infty}_k \mathscr H^2_{2k-1}} \le 2 \sum_{k=1}^{\infty} \left\lVert \Delta(W_{k}R_{k-1}) \right\rVert_{\operatorname{TC}}\le 2 \left\lVert \Delta(H) \right\rVert_{\operatorname{TC}} \text{ and } \\
&\sum_{k=1}^{\infty} \left\lVert \Delta(G_{2k-2}) \right\rVert_{\mathscr H^{\infty}_k \mathscr H^2_{2k-2}}   \le 2 \sum_{k=0}^{\infty} \left\lVert \Delta(W_{k}R_{k}) \right\rVert_{\operatorname{TC}}\le 2 \left\lVert \Delta(H) \right\rVert_{\operatorname{TC}}
\end{split}
\end{equation*}
which by summing up the two bounds yields the statement of the theorem.
\end{proof}
While Theorem \ref{singestim} controls the transfer functions, the subsequent theorem controls the actual dynamics from zero:
\begin{proof}[Proof of Theorem \ref{dyncorr}]
The operator norm of the control tensor is bounded by
\begin{equation*}
\begin{split}
\left\lVert U_k(s) \right\rVert &\le  \prod_{i=1}^k \left\lVert u(s_i) \right\rVert_{(\mathbb R^n, \left\lVert \bullet \right\rVert_{\infty})}\left\lVert \operatorname{id}_{\mathcal H} \otimes \sum_{i_1,...,i_k=1}^n \langle \widehat{e}_{i_1}\otimes ...\otimes \widehat{e}_{i_k}, \bullet \rangle \right\rVert \\
&\le  \prod_{i=1}^k \left\lVert u(s_i) \right\rVert_{(\mathbb R^n, \left\lVert \bullet \right\rVert_{\infty})}  \left\lVert \sum_{i_1,...,i_k=1}^n\langle \widehat{e}_{i_1}\otimes ...\otimes \widehat{e}_{i_k}, \bullet \rangle \cdot 1 \right\rVert   \le n^{k/2} \prod_{i=1}^k \left\lVert u(s_i) \right\rVert_{(\mathbb R^n, \left\lVert \bullet \right\rVert_{\infty})}
\end{split}
\end{equation*}
where we applied the Cauchy-Schwarz inequality to the product inside the sum to bound the $\ell^1$ norm by an $\ell^2$ norm.

It follows from \eqref{eq:G}, H\"older's inequality, and Minkowski's integral inequality that
\begin{equation*}
\begin{split}
&\left\lVert \Delta(C \varphi(t)) \right\rVert_{\mathbb R^m} \le  \sum_{k=1}^{\infty}\int_{\Delta_k(t)} \Bigg(\left\lVert U_k(s)\right\rVert_{\mathcal L(\mathbb R^m \otimes \mathbb R^{n^{\otimes k}},\mathbb R^m)} \cdot  \\
& \qquad \qquad \qquad \qquad \cdot  \left\lVert \sum_{i=1}^n\Delta \left(O_{k-1}(t-s_1,...,s_{k-1}-s_k)\psi_i \right)\otimes \widehat{e}_i   \right\rVert_{\mathbb R^m \otimes \mathbb R^{n^{\otimes k}}}  \Bigg)\ ds \\
&\le  \sum_{k=1}^{\infty}\int_{\Delta_k(t)}\underbrace{\left\lVert U_k(s)\right\rVert_{\mathcal L(\mathbb R^m \otimes \mathbb R^{n^{\otimes k}},\mathbb R^m)}}_{\le n^{k/2} \prod_{i=1}^k \left\lVert u(s_i) \right\rVert_{(\mathbb R^n, \left\lVert \bullet \right\rVert_{\infty})}} \left\lVert \Delta h_{k-1,0}(t-s_1,..,s_{k-1}-s_k) \right\rVert_{\operatorname{HS}\left(\mathbb R^n,\mathbb R^m \otimes \mathbb R^{n^{\otimes (k-1)}}\right)} \ ds \\
&\le  \sum_{k=1}^{\infty} \left( \left\lVert \Delta(h _{2k-1,0}) \right\rVert_{L^{1}_kL^{2}_{2k-1}(\operatorname{HS})}+  \left\lVert \Delta( h _{2k-2,0}) \right\rVert_{L^{1}_kL^{2}_{2k-2}(\operatorname{HS})} \right) \sqrt{n} \left\lVert u \right\rVert_{L^{\infty}((0,\infty),(\mathbb R^n, \left\lVert \bullet \right\rVert_{\infty}))}. 
\end{split}
\end{equation*}
Then, by \eqref{tracenorm}, Lemma \ref{Hankelestimate}, and the invariance property \eqref{eq:invp}
\begin{equation*}
\begin{split}
\left\lVert \Delta(C \varphi(t)) \right\rVert_{\mathbb R^m}& \le    \sum_{k=1}^{\infty}  \left\lVert \Delta(h _{k-1,k}) \right\rVert_{L^{1}_kL^{2}_{2k-1}(\operatorname{HS})} \sqrt{n} \left\lVert u \right\rVert_{L^{\infty}((0,\infty),(\mathbb R^n, \left\lVert \bullet \right\rVert_{\infty}))} \\
& \quad +\sum_{k=1}^{\infty}  \left\lVert \Delta(h _{k-1,k-1}) \right\rVert_{L^{1}_kL^{2}_{2k-2}(\operatorname{HS})}\sqrt{n} \left\lVert u \right\rVert_{L^{\infty}((0,\infty),(\mathbb R^n, \left\lVert \bullet \right\rVert_{\infty}))} \\
&\le  4  \sqrt{n} \left\lVert \Delta(H) \right\rVert_{\operatorname{TC}} \left\lVert u \right\rVert_{L^{\infty}((0,\infty),(\mathbb R^n, \left\lVert \bullet \right\rVert_{\infty}))}.
\end{split}
\end{equation*}
\end{proof}

\section{Applications}
\label{sec:spdes}
Throughout this section, we assume that we are given a filtered probability space $(\Omega, \mathcal F,(\mathcal F_t)_{t \ge T_0}, \mathbb P)$ satisfying the usual conditions, i.e. the filtration is right-continuous and $\mathcal F_{T_0}$ contains all $\mathcal F$ null-sets. 
We assume $X$ to be a \emph{real} separable Hilbert space. In the following subsection, we study an infinite-dimensional stochastic evolution equation with Wiener noise to motivate the extension of stochastic balanced truncation to infinite-dimensional systems that we introduce thereupon. We stick mostly to the notation introduced in the preceding sections and also consider the \emph{state-to-output (observation) operator} $C \in \mathcal L(X,\mathcal H)$, the \emph{control-to-state (control) operator} $Bu = \sum_{i=1}^n \psi_i u_i,$ and $A$ the generator of an exponentially stable $C_0$-semigroup $(T(t))$ on $X$. 

\subsection{Stochastic evolution equation with Wiener noise.}

Let $Y$ be a separable Hilbert space and $\operatorname{TC}(Y) \ni Q=Q^* \ge 0 $ a positive trace class operator.
We then consider a Wiener process $(W_t)_{t \ge T_0}$ \cite[Def. $2.6$]{GM11} adapted to the filtration $(\mathcal F_t)_{t \ge T_0}$ with covariance operator $Q$.

Furthermore, we introduce the Banach space $\left(\mathcal H_{2}^{(T_0,T)}(X), \sup_{t \in (T_0,T)}\left( \mathbb E \left(   \left\lVert Z_t \right\rVert_{X}\right)^2\right)^{1/2}\right)$ of jointly measurable $((T_0,T)\times \Omega \ni (t,\omega) \mapsto Z_t(\omega))$, $X$-valued processes adapted to the filtration $(\mathcal F_t)_{t \ge T_0}$ and mappings\footnote{We will drop an argument whenever it is convenient and at no risk of confusion. For instance, we will sometimes write $u(t)$ instead of $u(\omega,t)$ or omit the measure and $\sigma$-algebra such that $L^2(\Omega,\mathcal F_0,\mathbb P,X)$ is just denoted as $L^2(\Omega,X)$.}
\begin{equation*}
\begin{split}
& N \in \mathcal L(X, \mathcal L(Y,X)) \text{ and controls } u \in L^2_{\text{ad}}(\Omega_{\mathbb R_{\ge 0}}, \mathbb R^n) \cap L^{\infty}_{\text{ad}}(\Omega_{\mathbb R_{\ge 0}}, \mathbb R^n)
\end{split}
\end{equation*}
 where we recall the notation $\Omega_X:= \Omega \times X.$
 For the stochastic partial differential equation
\begin{equation}
\begin{split}
\label{eq: spde}
dZ_t&= (AZ_t +Bu(t))\ dt + N(Z_t) \ dW_t, \quad t >0  \\ 
 Z_{0}&=\xi \in L^2(\Omega,X)
\end{split}
\end{equation}
there exists by \cite[Theorem $3.5$]{GM11} a unique continuous mild solution in $\mathcal H_{2}^{(T_0,T)}$, satisfying $\mathbb P$-a.s. for $t \in [0,T]$
\begin{equation}
\label{eq:stochmildsol}
Z_t = T(t)\xi + \int_0^t T(t-s) Bu(s) \ ds + \int_0^t T(t-s) N(Z_s) \ dW_s.
\end{equation}

We refer to \eqref{eq: spde} with $B \equiv 0$ as the homogeneous part of that equation.
For solutions $Z_t^{\text{hom}}$ to the homogeneous part of \eqref{eq: spde} starting at $t=0$, let $\Phi(\bullet): L^2(\Omega,\mathcal F_0,X) \rightarrow \mathcal H_{2}^{(0,T)}(X)$ be the flow defined by the mild solution, i.e. $\Phi(t)\xi:=Z_t^{\text{hom}}$. If the initial time is some $T_0$ rather than $0$ we denote the (initial time-dependent) flow by $\Phi(\bullet,T_0):L^2(\Omega,\mathcal F_{T_0},X) \rightarrow \mathcal H_{2}^{(T_0,T)}(X)$. 
The ($X$-)adjoint of the flow is defined by $\langle \Phi(\bullet, T_0)\varphi_1,\varphi_2 \rangle_X = \langle \varphi_1,\Phi(\bullet, T_0)^*\varphi_2 \rangle_X$ for arbitrary $\varphi_1,\varphi_2 \in X.$
\begin{defi}[Exponential stability in m.s.s.]
The solution to the homogeneous system with flow $\Phi$ is called \emph{exponentially stable} in the mean square sense (m.s.s.) if there is some $c>0$ such that for all $\varphi_0 \in X$ and all $t \ge 0$
\begin{equation}
\label{eq: mss}
\mathbb E \left( \left\lVert \Phi(t) \varphi_0 \right\rVert^2_X \right) \lesssim e^{-c t}  \left\lVert \varphi_0 \right\rVert_X^2.
\end{equation}
\end{defi}

Lyapunov techniques to verify exponential stability for SPDEs of the form \eqref{eq: spde} are discussed in \cite[Section $6.2$]{GM11}.

We then define the variation of constants process $Y$ of the flow $\Phi$ as
\begin{equation}
\label{eq:DuHamTra}
Y_t(u) :=\int_0^t \Phi(t,s) Bu(s)\ ds= \sum_{i=1}^n\int_0^t  \Phi(t,s)\psi_i u_i(s) \ ds.
\end{equation}

This variation of constants process coincides with the mild solution to the full SPDE \eqref{eq: spde} almost surely for initial-conditions $\xi=0$. This follows from \eqref{eq:stochmildsol} and the stochastic Fubini theorem \cite[Theorem $2.8$]{GM11}, since the absolute integral exists by \eqref{eq: mss} and exponential stability of the semigroup, 
\begin{equation*}
\begin{split}
& \int_0^t T(t-s)  N (Y_s) \ dW_s = \int_0^t T(t-s) N\left(\int_0^s \Phi(s,r)B u(r) \ dr\right) \ dW_s  \\
&= \int_0^t T(t-s) N\left(\int_0^t \underbrace{\indic_{[0,s]}(r)}_{=\indic_{[r,t]}(s)}  \Phi(s,r)B u(r) \ dr\right) \ dW_s  \\
&= \int_0^t \left( \int_0^t T(t-s) N\left( \indic_{[r,t]}(s)  \Phi(s,r)B u(r)\right) \ dW_s\right) \ dr  \\
&= \int_0^t  \left(\int_r^t T(t-s)N\left(\Phi(s,r)B u(r)\right) \ dW_s \right)\ dr\\
&= \int_0^t \Phi(t,r)B u(r)- T(t-r)B u(r) \ dr = Y_t - \int_0^t T(t-r) Bu(r) \ dr
\end{split}
\end{equation*}
which can be rewritten as $Y_t =   \int_0^t T(t-s)N(Y_s) \ dW_s + \int_0^t T(t-r) Bu(r) \ dr.$

Another important property of the homogeneous solution to \eqref{eq: spde} is that it satisfies the \emph{homogeneous Markov property} \cite[Section $3.4$]{GM11}.
While the flow $\Phi$ is time-dependent as the SPDE is non-autononomous, there is an associated $ C_b$-Markov semigroup $P(t):  C_b(X) \rightarrow  C_b(X)$ satisfying $P(t)f(x)=\mathbb E(f(\Phi(s+t,s)x))$ independent of $s \ge 0$ and $P(t+s)f=P(t)P(s)f.$

The $ C_b$-Feller property, i.e.~$P(t)$ maps $ C_b(X)$ again into $ C_b(X),$ will not be needed in our subsequent analysis, but reflects the continuous dependence of the solution \eqref{eq: spde} on initial data.

In particular, we use that the $ C_b$-Markov semigroup can be extended to all $f$ for which the process is still integrable, i.e. $f(\Phi(t,s)x) \in L^1(\Omega,\mathbb R) $ for arbitrary $s \le t$ and $x \in X.$ 

By applying the Markov property to the auxiliary functions $f_{x,y}$ 
\begin{equation*}
\begin{split}
&\left\langle \Phi(T-t+s,s)^*x,BB^*\Phi(T-t+s,s)^*y \right\rangle_{\mathbb R^n}\\
&= \sum_{i=1}^n \underbrace{\left\langle  \Phi(T-t+s,s)\psi_i,y\right\rangle \left\langle x, \Phi(T-t+s,s)\psi_i\right\rangle}_{=:f_{x,y}(\Phi(T-t+s,s)\psi_i)}
 \end{split}
 \end{equation*}
with $0 \le t \le T$, $x,y \in X$, and $0 \le s \le T-t$ it follows by evaluating $\mathbb E(f_{x,y}(\Phi(T-t+s,s)\psi_i))$ at $s=0$ and $s=t$ that
\begin{equation}
\label{eq:markov}
 \mathbb E\left\langle \Phi(T-t,0)^*y,BB^*\Phi(T-t,0)^*x \right\rangle_{\mathbb R^n}=\mathbb E\left\langle \Phi(T,t)^*y,BB^*\Phi(T,t)^*x \right\rangle_{\mathbb R^n}.
 \end{equation}

In the following subsection we introduce a \emph{generalized stochastic balanced truncation framework} for systems with properties similar to the ones that we just discussed for the particular stochastic evolution equation \eqref{eq: spde}.

\subsection{Generalized stochastic balanced truncation}
For an exponentially stable flow $\Phi$ we define the stochastic observability map $W$ and reachability map $R$
\begin{equation}
\begin{split}
\label{eq:stoobreach}
& W \in \mathcal L (X, L^2(\Omega_{(0,\infty)},\mathcal H))\text{ with }(Wx)(t,\omega):= C \Phi(t,\omega) x  \text{ and }\\
& R \in \operatorname{HS}(L^2(\Omega_{(0,\infty)},\mathbb R^n), X)\text{ with } Rf:=\mathbb E \left(\int_{(0,\infty)}     \sum_{i=1}^n\Phi(s)\psi_i \langle f(s), \widehat{e}_i \rangle  \ ds \right).
\end{split}
\end{equation}
\begin{rem}
\label{rem:TC}
Let $\mathcal H \simeq \mathbb R^m,$ then each map $x \mapsto \langle \widehat{e}_i,Wx \rangle$ is a Carleman operator and by the characterization of Carleman operators of Hilbert-Schmidt type \cite[Theorem $6.4$ (iii)]{Wei} the operator $W$ is a Hilbert-Schmidt operator as well.
\end{rem}
Using the observability and reachability maps \eqref{eq:stoobreach}, we define stochastic \emph{observability} $\mathscr{O}=W^*W \in \mathcal L(X)$ and \emph{reachability} $\mathscr{P}=RR^* \in \operatorname{TC}(X)$ gramians satisfying for all $x,y \in X$
\begin{equation}
\begin{split}
\label{eq:gram}
\langle x, \mathscr{O}y \rangle &=\mathbb{E} \left( \int_0^{\infty} \langle C \Phi(t)x, C \Phi(t) y \rangle \ dt \right) \\ 
\langle x, \mathscr{P}y \rangle &=\mathbb{E} \left( \int_0^{\infty}  \langle B^* \Phi(t)^* x, B^* \Phi(t)^* y \rangle    \ dt \right).
\end{split}
\end{equation}

To obtain a dynamical interpretation of the gramians, let us recall that for compact self-adjoint operators $K: X\rightarrow X$, we can define the (possibly unbounded) Moore-Penrose pseudoinverse as
\begin{equation*}
\begin{split}
K^{\#}&:\operatorname{ran}(K) \oplus \operatorname{ran}( K)^{\perp} \subset X \rightarrow X \text{ such that }K^{\#}x := \sum_{\lambda \in \sigma(K)\backslash\{0 \}} \lambda^{-1} \langle x,v_{\lambda} \rangle v_{\lambda} 
\end{split}
\end{equation*}
using any orthonormal eigenbasis $(v_{\lambda})_{\lambda \in \sigma(K)}$ associated with eigenvalues $\lambda$ of $K$.

Then, for any time $\tau>0$ one defines the input energy $E^{\tau}_{\text{input}}: X \rightarrow [0,\infty]$ and output energy $E^{\tau}_{\text{output}}: X \rightarrow [0,\infty]$ up to time $\tau$ as 
\begin{equation}
\begin{split}
\label{eq:energy}
E^{\tau}_{\text{input}}(x)&:= \inf_{u \in L^2((0,\infty),\mathbb R^n); \mathbb E(Y_{\tau}(u))=x} \int_0^{{\tau}} \left\lVert u(t) \right\rVert^2  \ dt \text{ and }\\
E^{\tau}_{\text{output}}(x)&:= \left\lVert C\Phi x \right\rVert^2_{L^2(\Omega_{(0,\tau)},\mathcal H)}
\end{split}
\end{equation}
where $Y_t$ is the variation of constants process of the flow defined in \eqref{eq:DuHamTra}. In particular, the expectation value $ \mathbb E(Y_{\tau}(u))$ appearing in the definition of the input energy is a solution to the deterministic equation
\begin{equation}
\label{eq:classicaleq}
\varphi'(t) =  T(t) \varphi(t) + Bu(t), \quad \varphi(0) = 0
\end{equation}
where $u \in L^2((0,\infty),\mathbb R^n)$ is a deterministic control.
The theory of linear systems implies that $x$ is then reachable, by the dynamics of \eqref{eq:classicaleq}, after a fixed finite time ${\tau}>0$ if $x \in \operatorname{ran} \mathscr P_{\tau}^{\text{det}}$ where $ \mathscr P_{\tau}^{\text{det}}$ is the time-truncated deterministic linear gramian which for $x,y \in X$ is defined as
\[ \langle x,\mathscr P_{\tau}^{\text{det}}y \rangle := \int_0^{\tau} \langle B^*T(s)^*x,B^* T(s)^*y \rangle \ ds.\]
The control, of minimal $L^2$ norm, that steers the deterministic system \eqref{eq:classicaleq} into state $x$ after time $\tau$ is then given by $u(t)=\indic_{[0,\tau]}(t) B^*T(\tau-t)^*\left( \mathscr P_{\tau}^{\text{det}}\right)^{\#}x.$  
We also define time-truncated stochastic reachability and observability gramians $\mathscr P_{\tau}$ and $\mathscr O_{\tau}$ for $x,y \in X$
\begin{equation*}
\begin{split}
\langle x, \mathscr{P}_{\tau}y \rangle &=\mathbb{E} \left( \int_0^{\tau}  \langle B^*\Phi(t)^* x, B^* \Phi(t)^* y  \rangle    \ dt \right) \text{ and } \\
\langle x, \mathscr{O}_{\tau}y \rangle &=\mathbb{E} \left( \int_0^{\tau}  \langle C\Phi(t)x, C \Phi(t) y  \rangle    \ dt \right).
\end{split}
\end{equation*}
An application of the Cauchy-Schwarz inequality shows that $\operatorname{ker}(\mathscr{P}_{\tau}) \subset \operatorname{ker}(\mathscr{P}_{\tau}^{\text{det}})$ and thus $\overline{\operatorname{ran}}(\mathscr{P}_{\tau}^{\text{det}}) \subset \overline{\operatorname{ran}}(\mathscr{P}_{\tau}):$
\begin{equation*}
\begin{split}
\langle x, \mathscr{P}_{\tau}^{\text{det}} x \rangle &=\int_0^{\tau}  \left\lVert  B^*T(t)^* x \right\rVert^2    \ dt=\int_0^{\tau}  \left\lVert  \mathbb E ( B^*\Phi(t,0)^* x )\right\rVert^2    \ dt \\
&\le \mathbb E \int_0^{\tau}  \left\lVert   B^*\Phi(t,0)^* x \right\rVert^2    \ dt  = \langle x, \mathscr{P}_{\tau} x \rangle.
\end{split}
\end{equation*}
Since for $\tau_1>\tau_2: \operatorname{ker}(\mathscr{P}_{\tau_1})\subset  \operatorname{ker}(\mathscr{P}_{\tau_2})$ it also follows that $ \overline{\operatorname{ran}}(\mathscr{P}_{\tau_2}) \subset  \overline{\operatorname{ran}}(\mathscr{P}_{\tau_1}).$
Then, one has, as for finite-dimensional systems \cite[Prop. 3.10]{BR15}, the following bound on the input energy \eqref{eq:energy}:
\begin{lemm}
Let $x$ be a reachable by the flow defined in \eqref{eq:classicaleq} and $x \in \operatorname{ran}(\mathscr P_{\tau})$ then \[E^{\tau}_{\text{input}}(x)=\left\langle x,\left(\mathscr P_{\tau}^{\operatorname{det}}\right)^{\#}x \right\rangle  \ge \langle x, \mathscr P_{\tau}^{\#} x \rangle.\]
The output energy of any state $x \in X$ satisfies 
\[E^{\tau}_{\text{output}}(x)=\left\langle x, \mathscr O_{\tau} x \right\rangle\le \left\langle x, \mathscr O x \right\rangle. \]
\end{lemm}
\begin{proof}
The representation of the output energy is immediate from the definition of the (time-truncated) observability gramian. 
For the representation of the input energy we have by assumption $x \in \operatorname{ran}(\mathscr P_{\tau}^{\text{det}}) \cap \operatorname{ran}(\mathscr P_{\tau})$. Consider then functions 
\begin{equation*}
\begin{split}
u(t)&:=\indic_{[0,\tau]}(t) B^*T(\tau-t)^* \left( \mathscr P_{\tau}^{\text{det}}\right)^{\#}x \text{ and }
v(t):=\indic_{[0,\tau]}(t) B^*\Phi(\tau,t)^*\mathscr P_{\tau}^{\#}x.
\end{split}
\end{equation*}
Hence, we find since $x= \mathscr P_{\tau}^{\text{det}}\left( \mathscr P_{\tau}^{\text{det}}\right)^{\#}x=\mathscr P_{\tau}\mathscr P_{\tau}^{\#} x$ that
\begin{equation*}
\begin{split}
\mathbb E\int_0^{\tau} \left\langle v(s),u(s)-v(s) \right\rangle_{\mathbb R^n}\ ds=0
\end{split}
\end{equation*}
which implies the claim on the (time-truncated) reachability gramian
\begin{equation*}
\begin{split}
\left\langle x,\left(P_{\tau}^{\operatorname{det}}\right)^{\#}x \right\rangle
&=\mathbb E \int_0^{\tau} \left\lVert u(s)  \right\rVert_{\mathbb R^n}^2  \ ds\\
&=  \mathbb E \int_0^{\tau} \left\lVert v(s)  \right\rVert_{\mathbb R^n}^2  \ ds + \mathbb E \int_0^{\tau} \left\lVert u(s)-v(s)  \right\rVert_{\mathbb R^n}^2  \ ds \\
&\ge \left\langle x,P_{\tau}^{\#}x \right\rangle.
\end{split}
\end{equation*}

\end{proof}

\begin{defi}
The stochastic Hankel operator is defined as
\begin{equation}
\begin{split}
\label{eq: Hankelsto}
&H  \in \operatorname{HS}\left(L^2(\Omega_{(0,\infty)},\mathbb R^n),L^2(\Omega_{(0,\infty)},\mathcal H)\right) \text{ such that }(Hf)(t,\omega) = (WRf)(t,\omega).
\end{split}
\end{equation}
By Remark \ref{rem:TC} the Hankel operator is trace class if $\mathcal H\simeq \mathbb R^m$ for some $m \in \mathbb N.$
\end{defi}

From standard properties of the stochastic integral it follows that the expectation value of the solution $\mathbb E(Z_t)$ or $\mathbb E(CZ_t)$ to \eqref{eq: spde2} is just the solution $\varphi$ or $C\varphi$ to the linear and deterministic equation $\varphi'(t)=A\varphi(t)+B\mathbb Eu(t).$ We can then show Proposition \ref{theo2} which extends this analogy between stochastic and linear systems to the error bounds for deterministic controls:

\begin{proof}[Proposition \ref{theo2}]
Let $(e_n)$ and $(f_n)$ be orthonormal systems in $L^2((0,\infty), \mathbb{R}^n)$ and $L^2((0,\infty),\mathbb{R}^m)$, then they are also orthonormal in $L^2(\Omega_{(0,\infty)},\mathbb R^n)$ and $L^2(\Omega_{(0,\infty)},\mathbb R^m)$.  

Let $q_{k}(x):=\left\langle \widehat{e}_k,Cx  \right\rangle_{\mathbb R^m}$ and $g(\sigma):=\Delta\left(\mathbb E \left(C\Phi(\sigma)B \right) \right)\in \mathbb{R}^{m \times n}.$
From the definition of the trace norm \eqref{tracenorm} and the semigroup property it follows that 
\begin{equation*}
\begin{split}
&\left\lVert \Delta(H) \right\rVert_{\operatorname{TC}} \ge \sum_{i \in \mathbb{N}}  \left\lvert \langle f_i,\Delta(H) e_i \rangle \right\rvert\\
&= \sum_{i \in \mathbb{N}} \left\lvert \int_{(0,\infty)^2} \sum_{k=1}^m\Delta\left( \int_{\Omega^2} \left\langle f_i(s),\widehat{e}_k \right\rangle  \left\langle \widehat{e}_k, C\Phi(s,\omega')\Phi(t,\omega)B  \  e_i(t)  \right\rangle  \ d\mathbb P(\omega') \ d\mathbb P(\omega) \right) \ ds \ dt \right\rvert\\
&=\sum_{i \in \mathbb{N}} \left\lvert \int_{(0,\infty)^2} \sum_{j=1}^n  \sum_{k=1}^m  \left\langle f_i(s),\widehat{e}_k \right\rangle \Delta \left( \mathbb E((P(s)q_k)(\Phi(t)\psi_j))        \right) \langle \widehat{e}_j,e_i(t)\rangle  \ ds \ dt \right\rvert.
\end{split}
\end{equation*}

Then by the semigroup property of the time-homogeneous Markov process it follows that 
 \[\mathbb E((P(s)q_k)(\Phi(t)\psi_j))=(P(t)P(s)q_k)(\psi_j)=(P(t+s)q_k)(\psi_j)\]
 and thus
\begin{equation*}
\begin{split}
&\left\lVert \Delta(H) \right\rVert_{\operatorname{TC}} \ge \sum_{i \in \mathbb{N}} \left\lvert \int_{(0,\infty)^2} \sum_{j=1}^n\left\langle f_i(s),\Delta( C P(s+t)\psi_j) \langle e_i(t), \widehat{e}_j\rangle       \right\rangle_{\mathbb R^m}  \ ds \ dt \right\rvert\\
&= \sum_{i \in \mathbb{N}} \left\lvert \int_{(0,\infty)^2} \left\langle f_i(s),\Delta\left( \mathbb E\left(C \Phi(s+t) B \right)   \right) e_i(t)  \right\rangle_{\mathbb R^m} \ ds \ dt \right\rvert\\
&= \sum_{i \in \mathbb{N}} \left\lvert \int_{(0,\infty)^2} \left\langle f_i(s),g(s+t) \ e_i(t)  \right\rangle_{\mathbb R^m} \ ds \ dt \right\rvert.
\end{split}
\end{equation*}
The standard estimate for linear systems \cite[Theorem $2.1$]{CGP} implies then
\[\left\lVert \Delta(H) \right\rVert_{\operatorname{TC}}\ge \frac{1}{2} \left\lVert g \right\rVert_{L^1((0, \infty),\mathcal L(\mathbb{R}^n, \mathbb{R}^m))}.\]
Using this inequality the statement of the theorem follows from the homogeneity of the Markov semigroup and Young's inequality
\begin{equation*}
\begin{split}
& \left\lVert \mathbb E  \Delta(CY_{\bullet}(u)) \right\rVert_{L^p((0,\infty),\mathbb R^m)} \le   \left\lVert \int_{(0,\bullet)}  \left\lVert \Delta \left(  \mathbb E(C \Phi(\bullet-s) B)\right)  u(s) \right\rVert_{\mathbb R^m} \ ds  \right\rVert_{L^p((0,\infty),\mathbb R)} \\
 &\le \left\lVert \int_{(0,\infty)}  \left\lVert \indic_{(0,\infty)}(\bullet-s)\Delta \left(  \mathbb E(C \Phi(\bullet-s) B)\right)\right\rVert_{\mathcal L(\mathbb{R}^n,\mathbb{R}^m)}    \left\lVert \indic_{(0,\infty)}(s)  u(s) \right\rVert_{\mathbb R^m} \ ds  \right\rVert_{L^p((0,\infty),\mathbb R)} \\
 &\le    \left\lVert   \left\lVert \indic_{(0,\infty)} \Delta \left(  \mathbb E(C \Phi B)\right) \right\rVert_{\mathcal L(\mathbb{R}^n,\mathbb{R}^m)} *   \left\lVert \indic_{(0,\infty)} u \right\rVert_{\mathbb R^n} \right\rVert_{L^p((0,\infty),\mathbb R)}  \\
&\le   \left\lVert  g  \right\rVert_{L^1((0,\infty), \mathcal L(\mathbb{R}^n,\mathbb{R}^m))} \left\lVert u \right\rVert_{L^{p}((0,\infty), \mathbb{R}^n)}   \le  2\left\lVert \Delta(H) \right\rVert_{\operatorname{TC}}\left\lVert u \right\rVert_{L^{p}( (0,\infty), \mathbb{R}^n)}.
\end{split}
\end{equation*}
\end{proof}
\begin{defi}
The Volterra kernel of the stochastic Hankel operator is defined as 
\[ h((s,\omega),(t,\omega')):= C\Phi(s,\omega)\Phi(t,\omega')B\]
and the \emph{compressed Volterra kernel} reads $\overline{h}(s,\omega):= C\Phi(s,\omega)B.$
\end{defi}
While the error bound in Proposition \ref{theo2} relied essentially on linear theory, our next estimate stated in Theorem \ref{theo3} bounds the expected error. The proof strategy resembles the proof given for bilinear systems in Lemma \ref{Hankelestimate}. We commence, as we did for bilinear systems, by introducing the Volterra kernels of the stochastic Hankel operator.

\begin{proof}[Proof of Theorem \ref{theo3}]
We will show that the difference of compressed Volterra kernels $\overline h$ for both systems satisfies
\begin{equation}
\label{eq:tbound}
 \int_{0}^{\infty} \left\lVert  \Delta( \overline{h}(v,\bullet)) \right\rVert_{L^2(\Omega, \operatorname{HS}(\mathbb R^n, \mathbb R^m))} \ dv \le 2 \left\lVert \Delta(  H) \right\rVert_{\operatorname{TC}\left(L^2(\Omega_{(0,\infty)},\mathbb R^n),L^2(\Omega_{(0,\infty)},\mathbb R^m)\right)}.
\end{equation}
We start by showing how \eqref{eq:tbound} implies \eqref{eq:estimate}
\begin{equation*}
\begin{split}
\sup_{t \in (0,\infty)} \mathbb E  \left\lVert  \Delta(CY_t(u)) \right\rVert_{\mathbb R^m} &\le  \sup_{t \in (0,\infty)}  \int_{(0,t)}   \mathbb E \left\lVert \Delta \left(  C \Phi(t,s) B\right)  u(s) \right\rVert_{\mathbb R^m} \ ds   \\
&\le  \sup_{t \in (0,\infty)}  \int_{(0,t)}   \left(\mathbb E \left\lVert \Delta \left(  C \Phi(t,s) B\right)\right\rVert_{\mathcal L(\mathbb R^n, \mathbb R^m)}^2\right)^{1/2} \left(\mathbb E \left\lVert u(s) \right\rVert_{\mathbb R^n}^2\right)^{1/2} \ ds   \\
&\le   \int_{(0,\infty)}   \left(\mathbb E \left\lVert \Delta \left(  C \Phi(t) B\right)\right\rVert_{\mathcal L(\mathbb R^n, \mathbb R^m)}^2\right)^{1/2} \ dt \left\lVert  u \right\rVert_{\mathcal H_2^{(0,\infty)}(\mathbb R^n)}    \\
&  \le  2\left\lVert \Delta(H) \right\rVert_{\operatorname{TC}}\left\lVert u \right\rVert_{\mathcal H_2^{(0,\infty)}(\mathbb R^n)}.
\end{split}
\end{equation*}
Thus, it suffices to verify \eqref{eq:tbound}. Let $Z:= L^2\left(\Omega,\mathbb{R}^{m}\right)\otimes L^2 \left(\Omega,\mathbb{R}^{n}\right)$. The independence assumption in the theorem has been introduced for
\[\left\lVert \Delta(h((s,\bullet),(t,\bullet')))\right\rVert_{Z} = \left\lVert \Delta( \overline{h}(s+t,\bullet))\right\rVert_{L^2(\Omega, \operatorname{HS}(\mathbb R^n, \mathbb R^m))}\,\]
to hold. 
To see this, we consider an auxiliary function $\xi_{i}(x_1,x_2):=\big( \langle \widehat{e}_i,Cx_1-\widetilde Cx_2 \rangle_{\mathbb R^m} \big)^2$ where $C$ and $\widetilde C$ are the observation operators of the two systems.
By the independence assumption, there is again a Markov semigroup $(P(t))_{t \ge 0}$ associated with the time-homogeneous Markov process determined by the vector-valued flow $(\Phi(t),\widetilde{\Phi}(t))_{t \ge 0}$ such that $(P(t)\xi_i)(x_1,x_2):=\mathbb E(\xi_i(\Phi(s+t,s)x_1, \widetilde{\Phi}(s+t,s)x_2)).$  Let $(\psi_j)_{j \in \left\{1,..,n \right\}}, (\widetilde{\psi})_{j \in \left\{1,..,n \right\}}$ be the vectors in $X$ comprising the control operators $B$ and $\widetilde B$, respectively. The semigroup property of $(P(t))_{t \ge 0}$ implies then
\begin{equation} 
\begin{split}
\label{eq:condensation}
&\left\lVert \Delta(h((s,\bullet),(t,\bullet')))\right\rVert^2_{Z} \\
&= \sum_{i=1}^m \sum_{j=1}^n\int_{\Omega \times \Omega} \xi_i \left(\Phi(s,\omega)\Phi(t,\omega')\psi_j,\widetilde{\Phi(s,\omega)} \widetilde{\Phi(t,\omega')}\widetilde{\psi_j} \right) \ d \mathbb P(\omega) \ d \mathbb P(\omega')\\ 
&= \sum_{i=1}^m \sum_{j=1}^n\mathbb E\left(P(s) \xi_i\left(\Phi(t)\psi_j,\widetilde{\Phi(t)}\widetilde{\psi_j}\right)\right)\\ 
&= \sum_{i=1}^m \sum_{j=1}^n (P(t)P(s)\xi_{i})(\psi_j,\widetilde{\psi}_j) =  \sum_{i=1}^m \sum_{j=1}^n (P(s+t)\xi_{i})(\psi_j,\widetilde{\psi}_j) \\
&=  \left\lVert \Delta(\overline{h}(s+t,\bullet))\right\rVert_{L^2(\Omega, \operatorname{HS}(\mathbb R^n, \mathbb R^m))}^2.
\end{split}
\end{equation}
Let $M$ be large enough such that $\frac{1}{2} \int_{(2M,\infty)}  \left\lVert  \Delta( \overline{h}(v,\bullet)) \right\rVert_{L^2(\Omega, \operatorname{HS}(\mathbb R^n, \mathbb R^m))} \ dv  \le \varepsilon.$ Then, consider the integral function defined for $0<\alpha/2<x$
\[G(x,\alpha):=\frac{1}{\alpha} \int_{x-\alpha/2}^{x+\alpha/2} \Delta( \overline{h}(2s,\bullet)) \ ds.\]
By Lebesgue's differentiation theorem for Bochner integrals this function converges for $x \in (0,M)$ pointwise on a set $I \subset (0,M)$ of full measure to its integrand evaluated at $s=x$ as $\alpha \downarrow 0.$ 
In particular, for any $x \in I$ there is $\delta_x<\text{min}(x,M-x)$ such that if $0 <  \alpha /2 \le \delta_x$ then
\begin{equation}
\begin{split}
\label{eq: uniformly}
&\left\lvert \frac{1}{\alpha} \int_{x-\alpha/2}^{x+\alpha/2} \left\lVert  \Delta( \overline{h}(2s,\bullet) )\right\rVert_{L^2(\Omega, \operatorname{HS}(\mathbb R^n, \mathbb R^m))} \ ds- \left\lVert  \Delta( \overline{h}(2x,\bullet) )\right\rVert_{L^2(\Omega, \operatorname{HS}(\mathbb R^n, \mathbb R^m))} \right\rvert \\
&  \le \frac{1}{\alpha} \int_{x-\alpha/2}^{x+\alpha/2} \left\lVert  \Delta( \overline{h}(2s,\bullet) - \overline{h}(2x,\bullet) )\right\rVert_{L^2(\Omega, \operatorname{HS}(\mathbb R^n, \mathbb R^m))} \ ds \le \varepsilon/M. 
\end{split}
\end{equation}
Since $\Delta(h((s,\bullet),(t,\bullet')))$ contains the products of two flows, the function $\Delta(h((x,\bullet),(x,\bullet')))$ is a.e. well-defined on the diagonal. 
Then, there is a set $J$ of full measure such that every $x \in J \subset (0,M)$ is a Lebesgue point of the Volterra kernel on the diagonal. Thus, as for the condensed Volterra kernel above, there is also for the full Volterra kernel some $0<\gamma_x<\text{min}(x,M-x)$ such that if $0 < \alpha/2 \le \gamma_x$ then
\begin{equation}
\label{eq:epsM}
\frac{1}{\alpha^2} \int_{x-\alpha/2}^{x+\alpha/2}  \int_{x-\alpha/2}^{x+\alpha/2}  \left\lVert \Delta(h((s,\bullet),(t,\bullet')))- \Delta(h((x,\bullet),(x,\bullet')))\right\rVert_Z  \ ds \ dt \le \varepsilon/M.
\end{equation}

This is due to Lebesgue's differentiation theorem for Banach space-valued integrands applied to the flows $\Phi,\widetilde{\Phi}$ and the following estimate
\begin{equation*}
\begin{split}
&\frac{1}{\alpha^2} \int_{x-\alpha/2}^{x+\alpha/2}  \int_{x-\alpha/2}^{x+\alpha/2}  \left\lVert \Delta\left(h((s,\bullet),(t,\bullet'))\right)- \Delta\left(h((x,\bullet),(x,\bullet'))\right)\right\rVert_Z \ ds \ dt \\
&\le \frac{1}{\alpha^2} \int_{x-\alpha/2}^{x+\alpha/2}  \int_{x-\alpha/2}^{x+\alpha/2} \left\lVert \Delta\left(h((s,\bullet),(t,\bullet'))\right)- \Delta\left(h((s,\bullet),(x,\bullet'))\right)\right\rVert_Z \ ds \ dt \\
&\ + \frac{1}{\alpha^2} \int_{x-\alpha/2}^{x+\alpha/2}  \int_{x-\alpha/2}^{x+\alpha/2}  \left\lVert \Delta\left(h((s,\bullet),(x,\bullet'))\right)- \Delta\left(h((x,\bullet),(x,\bullet'))\right)\right\rVert_Z  \ ds \ dt \\
&\le \frac{\left\lVert C \right\rVert \left\lVert B \right\rVert_{\operatorname{HS}}}{\alpha^2} \int_{x-\alpha/2}^{x+\alpha/2}  \int_{x-\alpha/2}^{x+\alpha/2} \left\lVert \Phi(s)\right\rVert_{L^2(\Omega, \mathcal{L}(X))}\left\lVert \Phi(t)-\Phi(x) \right\rVert_{L^2(\Omega, \mathcal{L}(X))} \ ds \ dt \\
&\ + \frac{\left\lVert C \right\rVert \left\lVert B \right\rVert_{\operatorname{HS}}}{\alpha^2} \int_{x-\alpha/2}^{x+\alpha/2}  \int_{x-\alpha/2}^{x+\alpha/2}  \left\lVert \Phi(s)-\Phi(x) \right\rVert_{L^2(\Omega, \mathcal{L}(X))} \left\lVert \Phi(x) \right\rVert_{L^2(\Omega, \mathcal{L}(X))}   \ ds \ dt \\
& + \frac{\left\lVert \widetilde{C} \right\rVert \left\lVert \widetilde{B}  \right\rVert_{\operatorname{HS}}}{\alpha^2} \int_{x-\alpha/2}^{x+\alpha/2}  \int_{x-\alpha/2}^{x+\alpha/2} \left\lVert \widetilde{\Phi(s)} \right\rVert_{L^2(\Omega, \mathcal{L}(X))}\left\lVert \widetilde{\Phi(t)}-\widetilde{\Phi(x)} \right\rVert_{L^2(\Omega, \mathcal{L}(X))} \ ds \ dt \\
&\ + \frac{\left\lVert \widetilde{C} \right\rVert \left\lVert \widetilde{B} \right\rVert_{\operatorname{HS}}}{\alpha^2} \int_{x-\alpha/2}^{x+\alpha/2}  \int_{x-\alpha/2}^{x+\alpha/2}  \left\lVert \widetilde{\Phi(s)}-\widetilde{\Phi(x)} \right\rVert_{L^2(\Omega, \mathcal{L}(X))} \left\lVert\widetilde{\Phi(x)}\right\rVert_{L^2(\Omega, \mathcal{L}(X))}   \ ds \ dt.
\end{split}
\end{equation*}

Consider then the family of intervals $I_x:=[x-\operatorname{min} \left(\delta_x, \gamma_x\right),x+\operatorname{min} \left(\delta_x, \gamma_x\right)]$ for $x \in I\cap J.$ Lebesgue's covering theorem \cite[Theroem $26$]{L10} states that, after possibly shrinking the diameter of the sets $I_x$ first, there exists an at most countably infinite family of disjoint sets $(I_{x_i})_{i \in \mathbb N}$ covering $I \cap J$ such that the Lebesgue measure of $I \cap J \cap \left( \bigcup_{i \in \mathbb N} I_{x_i} \right)^{C}$ is zero. Using additivity of the Lebesgue measure, there are for every $\varepsilon>0$ finitely many points $x_1,..,x_n \in I \cap J$ such that the set $I\cap J \cap \left( \bigcup_{i=1}^n I_{x_i} \right)^{C}$ has Lebesgue measure at most $\varepsilon$ . Thus, we have obtained finitely many disjoint sets $I_{x_i}$ of total measure $M-\varepsilon$ such that for $0<\alpha_i  /2 \le \operatorname{diam}(I_{x_i})/2$ both estimates \eqref{eq: uniformly} and \eqref{eq:epsM} hold at $x=x_i$ where $x_i$ is the midpoint of $I_{x_i}.$

For every $i \in \left\{1,..,n \right\}$ fixed, we introduce the family of sesquilinear forms $(L_{i})$
\begin{equation*} 
\begin{split}
\label{form2}
&L_{i}:  L^2\left(\Omega,\mathbb{R}^{m}\right) \oplus L^2 \left(\Omega,\mathbb{R}^{n}\right) \rightarrow \mathbb{R}  \\
& (f,g) \mapsto \int_{\Omega^2} \left\langle f(\omega),\Delta(h((x_i,\omega),(x_i,\omega'))) g(\omega') \right\rangle_{\mathbb{R}^{m}} \ d\mathbb P(\omega) \ d\mathbb P(\omega')
\end{split}  
\end{equation*}
 and for $Z:= L^2\left(\Omega,\mathbb{R}^{m}\right)\otimes L^2 \left(\Omega,\mathbb{R}^{n}\right)$ we can define a Hilbert-Schmidt operator of unit $\operatorname{HS}$-norm given by $Q_i : L^2 \left(\Omega,\mathbb{R}^{n}\right)  \rightarrow L^2\left(\Omega,\mathbb{R}^{m}\right)$
\begin{equation*}
\begin{split}
&(Q_i\varphi)(\omega):=\int_{\Omega} \frac{\Delta(h((x_i,\omega),(x_i,\omega')))}{\left\lVert \Delta(h((x_i,\bullet),(x_i,\bullet')))\right\rVert_{Z}} \varphi (\omega') \ d\mathbb P(\omega').
\end{split}
\end{equation*}
Doing a singular value decomposition of $Q_i$ yields orthonormal systems $f_{k,i} \in L^2\left(\Omega,\mathbb{R}^{m}\right), \ g_{k,i} \in L^2 \left(\Omega,\mathbb{R}^{n}\right)$ as well as singular values $\sigma_{k,i} \in [0,1]$ parametrized by $k \in \mathbb N.$ For any $\delta>0$ given there is $N(\delta)$ large enough such that
\[ \left\lVert \frac{\Delta(h((x_i,\bullet),(x_i,\bullet')))}{\left\lVert \Delta(h((x_i,\bullet),(x_i,\bullet')))\right\rVert_{Z}} - \sum_{k=1}^{N(\delta)} \sigma_{k,i} (   f_{k,i} \otimes g_{k,i} )\right\rVert_{Z}<\delta. \]

Thus, there are also $f_{k,i}\in L^2\left(\Omega,\mathbb{R}^{m}\right)$ and $g_{k,i}  \in L^2 \left(\Omega,\mathbb{R}^{n}\right)$ orthonormalized, $N_i \in \mathbb N,$ and $\sigma_{k,i} \in [0,1]$ such that 
\begin{equation}
\begin{split}
\label{eq: functionalbound2}
&\left\lvert \left\langle  \frac{\Delta(h((x_i,\bullet),(x_i,\bullet')))}{\left\lVert \Delta(h((x_i,\bullet),(x_i,\bullet')))\right\rVert_{Z}} -\sum_{k=1}^{N_i} \sigma_{k,i} ( f_{k,i}  \otimes  g_{k,i} ), \Delta(h((x_i,\bullet),(x_i,\bullet'))) \right\rangle_{Z} \right\rvert \\
&=\left\lvert \left\lVert \Delta(h((x_i,\bullet),(x_i,\bullet')))\right\rVert_{Z} -\sum_{k=1}^{N_i} \sigma_{k,i} L_{i}(f_{k,i},g_{k,i}) \right\rvert <\varepsilon / M. 
\end{split}
\end{equation}

Then, $s_{k,i}(s,\omega):= \tfrac{g_{k,i}(\omega)\indic_{I_{x_i}}(s)}{\sqrt{\left\lvert I_{x_i} \right\rvert}} $ and $t_{k,i}(s,\omega):= \tfrac{f_{k,i}(\omega) \indic_{I_{x_i}}(s)}{\sqrt{\left\lvert I_{x_i} \right\rvert}}  $ form orthonormal systems in $L^2\left(\Omega_{(0,\infty)},\mathbb{R}^n\right)$ and $L^2\left(\Omega_{(0,\infty)},\mathbb{R}^{m}\right)$ respectively, both in $k$ and $i$, such that for $\mathcal I_i:=\Omega_{I_{x_i}}\times \Omega_{I_{x_i}}$
it follows that
\begin{equation} 
\begin{split} 
\label{eq:char2}
&\langle t_{k,i}, \Delta( H) s_{k,i} \rangle_{L^2\left(\Omega_{(0,\infty)},\mathbb{R}^{m}\right)} \\
&=\tfrac{1}{\left\lvert I_{x_i} \right\rvert}\int_{\mathcal I_{i}} \left\langle f_{k,i}(\omega),\Delta(h((s,\omega),(t,\omega')))g_{k,i}(\omega')\right\rangle_{\mathbb{R}^{m}} \ dt \ ds \ d\mathbb P(\omega) \ d \mathbb P(\omega').
\end{split} 
\end{equation}
Hence, we get
\begin{equation*} 
\begin{split}
&\left\lvert \sum_{i=1}^{n } \left(\sum_{k=1}^{N_i} \sigma_{k,i} \langle   t_{k,i}, \Delta( H) s_{k,i} \rangle_{L^2\left(\Omega_{(0,\infty)},\mathbb{R}^{m}\right)}  -  \tfrac{1}{\left\lvert I_{x_i} \right\rvert} \int_{I_{x_i}^2}  \left\lVert  \Delta( \overline{h}(2x_i,\bullet) )\right\rVert_{L^2(\Omega, \operatorname{HS}(\mathbb R^n, \mathbb R^m))} \ ds \ dt\right) \right\rvert \\
&\le   \sum_{i=1}^{n} \tfrac{1}{\left\lvert I_{x_i} \right\rvert} \int_{I_{x_i}^2}  \Bigg(\left\lvert \sum_{k=1}^{N_i} \sigma_{k,i} \left\langle   g_{k,i} \otimes f_{k,i}, \left(\Delta\left(h((s,\bullet),(t,\bullet'))\right)-\Delta({h}((x_i,\bullet),(x_i,\bullet')))\right) \right\rangle_Z  \right\rvert \\
&\qquad \qquad \qquad   +\left\lvert  \sum_{k=1}^{N_i} \sigma_{k,i}  L_{i}(f_{k,i},g_{k,i})-\left\lVert \Delta(h((x_i,\bullet),(x_i,\bullet'))) \right\rVert_{Z} \right\rvert \\
& \qquad \qquad \qquad   +\left\lvert \sum_{k=1}^{N_i} \left\lVert  \Delta(h((x_i,\bullet),(x_i,\bullet'))) \right\rVert_{Z}- \left\lVert  \Delta( \overline{h}(2x_i,\bullet)) \right\rVert_{L^2(\Omega, \operatorname{HS}(\mathbb R^n, \mathbb R^m))}\right\rvert \Bigg) \ ds \ dt \ \lesssim  \varepsilon.
\end{split}
\end{equation*}
The bound on the first term follows from \eqref{eq:epsM} and
$\left\lVert \sum_{i=1}^{N_i} \sigma_{k,i}   g_{k,i} \otimes f_{k,i} \right\rVert_{Z} \le 1.$
The bound on the second term follows from \eqref{eq: functionalbound2} and the third term is \eqref{eq:condensation}.
We then compute further that
\begin{equation*}
\begin{split}
&\left\lvert \sum_{i=1}^{n } \left(\tfrac{1}{\left\lvert I_{x_i} \right\rvert} \int_{I_{x_i}^2}  \left\lVert  \Delta( \overline{h}(2x_i,\bullet) )\right\rVert_{L^2(\Omega, \operatorname{HS}(\mathbb R^n, \mathbb R^m))} \ ds \ dt - \tfrac{1}{2} \int_{2I_{x_i} }  \left\lVert  \Delta( \overline{h}(v,\bullet) )\right\rVert_{L^2(\Omega, \operatorname{HS}(\mathbb R^n, \mathbb R^m))}  \ dv \right) \right\rvert \\
&\le \left\lvert \sum_{i=1}^{n } \left(\tfrac{1}{\left\lvert I_{x_i} \right\rvert} \int_{I_{x_i}^2}  \left\lVert  \Delta( \overline{h}(2x_i,\bullet) )\right\rVert_{L^2(\Omega, \operatorname{HS}(\mathbb R^n, \mathbb R^m))} \ ds \ dt - \left\lvert I_{x_i} \right\rvert  \left\lVert  \Delta( \overline{h}(2x_i,\bullet) )\right\rVert_{L^2(\Omega, \operatorname{HS}(\mathbb R^n, \mathbb R^m))}   \right) \right\rvert \\
&+\left\lvert \sum_{i=1}^{n } \left(\left\lvert I_{x_i} \right\rvert  \left\lVert  \Delta( \overline{h}(2x_i,\bullet) )\right\rVert_{L^2(\Omega, \operatorname{HS}(\mathbb R^n, \mathbb R^m))} - \tfrac{1}{2} \int_{2I_{x_i}} \left\lVert \Delta( \overline{h}(v,\bullet) )\right\rVert_{L^2(\Omega, \operatorname{HS}(\mathbb R^n, \mathbb R^m))} \ dv   \right) \right\rvert \lesssim \varepsilon
\end{split}
\end{equation*}
where we used \eqref{eq: uniformly} to obtain the second estimate.
Combining the two preceding estimates, the Theorem follows from the characterization of the trace norm given in \eqref{tracenorm}.
\end{proof}

Next, we study conditions under which convergence of flows implies convergence of stochastic Hankel operators.
Let $(\Phi_i)$ be a sequence of flows converging in $L^2(\Omega_{(0,\infty)},\mathcal L(X))$ to $\Phi$ and $W_i$, $R_i$ the observability and reachability maps derived from $\Phi_i$ as in \eqref{eq:stoobreach}.
For the observability map this yields convergence in operator norm
\[ \left\lVert W-W_i \right\rVert^2=\mathbb E\int_{(0,\infty)} \left\lVert C(\Phi-\Phi_i)(t) \right\rVert_{\mathcal L(X,\mathcal H)}^2 \ dt \xrightarrow[i \rightarrow \infty]{} 0. \] 
If $\mathcal H \simeq \mathbb{R}^m$ then it follows by an analogous estimate that $W_i$ converges to $W$ in Hilbert-Schmidt norm, too \cite[Theorem $6.12$(iii)]{Wei}. 

For the reachability map we choose an ONB $(e_k)_{k \in \mathbb{N}}$ of $L^2(\Omega_{(0,\infty)},\mathbb R)$ which we extend by tensorisation $e_k^j:=e_k \otimes \widehat{e}_j$ for $j \in \left\{1,..,n \right\}$ to an ONB of  $L^2(\Omega_{(0,\infty)},\mathbb R^n).$ 
Using this basis and an orthonormal basis $(f_l)_{l \in \mathbb N}$ of $X$, it follows that
\begin{equation*}
\begin{split}
&\left\lVert R_i-R \right\rVert^2_{\operatorname{HS}(L^2(\Omega_{(0,\infty)},\mathbb R^n), X)}\\
&=\sum_{l \in \mathbb N}\sum_{k \in \mathbb{N}} \sum_{j=1}^n \left\lvert \int_{\Omega_{(0,\infty)}}\left\langle f_l, ( \Phi- \Phi_i)(t)(\omega)\psi_j \right \rangle_X e_k(t)(\omega)   \ dt \  d\mathbb P(\omega) \right\rvert^2 \\
&=\sum_{l \in \mathbb N} \sum_{j=1}^n \int_{\Omega_{(0,\infty)}} \left\lvert \left\langle f_l, ( \Phi- \Phi_i)(t)(\omega)\psi_j\right\rangle_X \right\rvert^2 \ dt  \ d\mathbb P(\omega)\\
&=\sum_{j=1}^n \int_{\Omega_{(0,\infty)}}  \left\lVert  ( \Phi- \Phi_i)(t)(\omega)\psi_j \right\rVert^2_X \ dt  \ d\mathbb P(\omega)  \xrightarrow[i \rightarrow \infty]{} 0.
\end{split}
\end{equation*}
As in the bilinear case, we obtain from this a convergence result for stochastic Hankel operators:
\begin{corr}
Let $H_i$ denote the Hankel operators associated with flows $\Phi_i$ converging in $L^2(\Omega_{(0,\infty)},\mathcal L(X))$ to $\Phi.$ Then, the $H_i$ converge in Hilbert-Schmidt norm to $H$ 
\[ \left\lVert H_i - H \right\rVert_{\operatorname{HS}} \le \left\lVert W_i-W \right\rVert \left\lVert R_i \right\rVert_{\operatorname{HS}} + \left\lVert W \right\rVert \left\lVert R_i- R \right\rVert_{\operatorname{HS}}\xrightarrow[i \rightarrow \infty]{} 0\]
and if $\mathcal H \simeq \mathbb{R}^m$ then the convergence is also in the sense of trace class operators
\[ \left\lVert H_i - H \right\rVert_{\operatorname{TC}} \le \left\lVert W_i-W \right\rVert_{\operatorname{HS}} \left\lVert R_i \right\rVert_{\operatorname{HS}} + \left\lVert W \right\rVert_{\operatorname{HS}} \left\lVert R_i- R \right\rVert_{\operatorname{HS}}\xrightarrow[i \rightarrow \infty]{} 0.\]
In particular, all singular values of $H_i$ converge to the singular values of $H$ \cite[Corollary $2.3$]{Krein} and, if the respective singular values non-degenerate, then the singular vectors converge in norm as well (see the proof of Lemma \ref{SVC}).
\end{corr}
To exhibit the connection between the model reduction methods for SPDEs and bilinear systems we finally state a weak version of the stochastic Lyapunov equations for real-valued L\'{e}vy-noise as stated for finite-dimensional systems in \cite[Eq. (14), (22)]{BR15}. Let $(L_t)$ be a square-integrable scalar L\'{e}vy process, then $M_t:=L_t-t \mathbb E(L_1)$ is a square-integrable centred martingale \cite[Theorem $2.7$]{BR15}. Its quadratic variation measure satisfies $d\langle M,M \rangle_t= \mathbb E\left(M_1^2\right) \ dt$. 
Let $(X_s)$ be an $X$-valued, predictable process with $\int_0^T \left\lVert X_s \right\rVert_X^2 d\langle M,M \rangle_s< \infty$ then the stochastic integral is defined by the unconditional convergent series $\int_0^t X_s \ dM_s:=\sum_{k \in \mathbb N} \int_0^t \langle X_s, e_k \rangle \ dM_s \ e_k $ where $(e_k)$ is any ONB of $X$ for $t \in [0,T]$ and the isometry formula
\[ \mathbb E \left\lVert \int_0^t X_s \ dM_s \right\rVert_X^2 = \mathbb E \int_0^t \left\lVert X_s \right\rVert_X^2 \ d\langle M,M \rangle_s  \]
holds \cite[Def. $6$ and Prop. $8$]{T13}. Moreover, from the series representation it follows from one-dimensional theory \cite[Theorem $2.11$]{BR15} that $\int_0^T X_s \ dM_s$ is a martingale and $\mathbb E\int_0^T X_s \ dM_s=0.$

Consider $n$ independent copies of such martingales $(M_t^{(j)})_{j \in \left\{1,\ldots,n \right\}}$ and the control operator $B$ as before.
We then study the stochastic evolution equation
\begin{equation}
\begin{split}
\label{eq: spde2}
dZ_t&= (AZ_t + Bu) \ dt +\sum_{j=1}^{n} N_j Z_t \ dM^{(j)}_t, \quad t >0  \\ 
 Z_{0}&=\xi
\end{split}
\end{equation}
for $\xi \in L^2(\Omega,\mathcal F_0,\mathbb P,X)$, $A$ the generator of a $C_0-$semigroup $(T(t))$, and $N_j \in \mathcal L(X)$. 
Then, the homogeneous part of \eqref{eq: spde2}, i.e. without the control term $Bu$, defines a unique predictable process $Z_t^{\text{hom}}:=\Phi(t)\xi \in \mathcal H_{2}^{(0,T)}$ \cite[Def. 9.11, Theorem $9.15$, Theorem $9.29$]{P15} with flow $\Phi$ that satisfies the homogeneous Markov property \cite[Prop. $9.31$ and $9.32$]{P15} and
\begin{equation}
\label{eq:mildsolpro}
Z_t^{\text{hom}}=T(t)\xi+ \sum_{j=1}^{n} \int_{0}^t T(t-s)N_j Z_s^{\text{hom}} \ dM^{(j)}_s.
\end{equation}
The adjoint equation to \eqref{eq: spde2} shall be defined with initial condition $Y_0 = \xi$ as
\begin{equation*}
\begin{split}
dY_t&= (A^*Y_t+ Bu) \ dt +\sum_{j=1}^{n} N_j^* Y_t \ dM^{(j)}_t , \quad t >0 
\end{split}
\end{equation*}
and the mild solution to the homogeneous part of this equation is
\begin{equation}
\label{eq:mildsolpro2}
Y_t^{\text{hom}}=T(t)^*\xi+ \sum_{j=1}^{n} \int_{0}^t T(t-s)^*N_j^* Y_s^{\text{hom}} \ dM^{(j)}_s.
\end{equation}
Let $\Psi$ be the flow of the adjoint equation such that $Y_t^{\text{hom}}:=\Psi(t)\xi$ then the $X$-adjoint of $\Psi$ satisfies the variation of constant formula
\[ \Psi(t)^*\xi= T(t)\xi+ \sum_{j=1}^{n} \int_{0}^t \Psi(s)^*N_j T(t-s)\xi \ dM^{(j)}_s.\]
For $\Phi$ being an exponentially stable flow in m.s.s. to \eqref{eq:mildsolpro}, we then define another observability gramian for \eqref{eq: spde2} by
\[ \langle x,\mathscr O^{\text{L\'{e}vy}} y \rangle:=\int_0^{\infty} \langle C \Psi(t)^* x,C \Psi(t)^*y \rangle \ dt.\] 
To see that $\mathscr O^{\text{L\'{e}vy}}$ coincides with the standard stochastic observability gramian \eqref{eq:gram} $\mathscr O$, we must show that for all $x \in X:$  $\mathbb E \left\lVert C \Phi(t)x \right\rVert_{\mathcal H}^2 = \mathbb E\left\lVert C \Psi(t)^*x \right\rVert_{\mathcal H}^2.$
Applying It\={o}'s isometry we obtain from \eqref{eq:mildsolpro} using sets $\Delta_k(t):=\{(s_1,...,s_k) \in \mathbb{R}^k; 0 \le s_k \le ...\le s_1 \le t\}$
\begin{equation*}
\begin{split}
& \mathbb E  \left\lVert C \Phi(t)x \right\rVert_{\mathcal H}^2 = \left\lVert  CT(t) x \right\rVert_{\mathcal H}^2+ \sum_{i=1}^n \mathbb E\left(M^{(i)}(1)^2\right) \ \mathbb E \int_0^t \left\lVert C T(t-s_1)N_{i} \Phi(s_1) x \right\rVert_{\mathcal H}^2 ds_1 \\
&=\left\lVert  CT(t) x \right\rVert_{\mathcal H}^2 + \sum_{k=1}^{\infty} \sum_{i_1,..,i_k=1}^n \prod_{j=1}^k \mathbb E\left(M^{(i_j)}(1)^2\right) \cdot\\
& \qquad \qquad \qquad \qquad \qquad \qquad  \cdot \int_{\Delta_k(t)} \left\lVert C T(t-s_1) \prod_{j=1}^{k-1}\left(N_{i_j} T(s_j-s_{j+1}) \right) N_{i_k}T(s_{k}) x \right\rVert_{\mathcal H}^2 ds  
   \end{split}
 \end{equation*}
 whereas it follows from \eqref{eq:mildsolpro2}
 \begin{equation*}
\begin{split}
&\mathbb E  \left\lVert C \Psi(t)^*x \right\rVert_{\mathcal H}^2= \left\lVert  CT(t) x \right\rVert_{\mathcal H}^2+ \sum_{i=1}^n \mathbb E \left(M^{(i)}(1)^2\right) \ \mathbb E \int_0^t \left\lVert C \Psi(s_1)^*N_i T(t-s_1) x \right\rVert_{\mathcal H}^2 ds_1 \\
&=\left\lVert  CT(t) x \right\rVert_{\mathcal H}^2 + \sum_{k=1}^{\infty} \sum_{i_1,..,i_k=1}^n \prod_{j=1}^k \mathbb E\left(M^{(i_j)}(1)^2\right) \cdot \\
& \qquad \qquad \qquad \qquad \qquad \qquad \cdot \int_{\Delta_k(t)}  \left\lVert C T(s_{k}) \prod_{j=k-1}^{1} \left(N_{i_{j+1}} T(s_{j}-s_{j+1}) \right) N_{i_1}T(t-s_1) x \right\rVert_{\mathcal H}^2 ds. 
 \end{split}
 \end{equation*}
An inflection of the integration domain shows then that both expressions (and hence the gramians) coincide.

Finally, the gramians satisfy the following Lyapunov equations for scalar L\'{e}vy-type noise (cf.~\cite{BR15} for the finite dimensional analogue):
\begin{lemm}
Let $\Phi$ be an exponentially stable flow in m.s.s. to \eqref{eq:mildsolpro} such that both gramians exist. Let $x_1,y_1 \in D(A^*)$ and $x_2,y_2 \in D(A),$ then 
\begin{equation*}
\begin{split}
&\langle x_1, BB^* y_1 \rangle + \langle A^*x_1, \mathscr P y_1 \rangle + \langle x_1, \mathscr P A^*y_1 \rangle +\sum_{j=1}^n \langle N_j^*x_1, \mathscr P N_j^*y_1 \rangle \ \mathbb{E}(M^{(j)}(1)^2) =0 \text{ and } \\
&\langle x_2, C^*C y_2 \rangle + \langle Ax_2, \mathscr O y_2 \rangle + \langle x_2, \mathscr O Ay_2 \rangle +\sum_{j=1}^n \langle N_jx_2, \mathscr O N_jy_2 \rangle \ \mathbb{E}(M^{(j)}(1)^2) =0.
\end{split}
\end{equation*}
\end{lemm}
\begin{proof}
For every $i \in \left\{1,...,n \right\}$ there is a weak formulation of the homogeneous solution to \eqref{eq: spde2} \cite[Theorem $9.15$]{P15}
\begin{equation*}
\begin{split}
\langle \Phi(t)\psi_i,x_1 \rangle&= \langle \psi_i,x_1 \rangle + \int_0^t \langle  \Phi(s) \psi_i,A^*x_1 \rangle \ ds + \sum_{j=1}^n \int_0^t \langle \Phi(s)\psi_i, N_j^*x_1  \rangle  \ dM^{(j)}_s.
\end{split}
\end{equation*}
Stochastic integration by parts yields after summing over $i \in \left\{1,..,n \right\}$
\begin{equation*}
\begin{split}
&\langle \Phi(t)^*x_1,BB^*  \Phi(t)^*y_1 \rangle 
= \langle x_1, BB^*y_1 \rangle +  \sum_{i=1}^n \int_0^t \langle \Phi(s) \psi_i,x_1 \rangle_{-} \ d\langle \Phi(s) \psi_i,y_1 \rangle \\
 &+ \sum_{i=1}^n \int_0^t \langle \Phi(s) \psi_i,y_1 \rangle_{-} \ d\langle \Phi(s)\psi_i,x_1 \rangle 
 + \sum_{i=1}^n \langle \langle x_1,\Phi(t) \psi_i \rangle, \langle \Phi(t)\psi_i ,y_1 \rangle \rangle_t
\end{split}
\end{equation*}
where the subscript $-$ indicates left-limits.

From the quadratic variation process \cite[Eq.(8)]{BR15}
 \[ \sum_{i=1}^{n}\mathbb E \langle \langle x_1,\Phi(t) \psi_i \rangle \rangle_t=\sum_{j=1}^{n}\mathbb E \left(\int_0^t  \langle \Phi(s)^*  N_j^*x_1,BB^*\Phi(s)^* N_j^*x_1  \rangle \ ds  \right) \mathbb E(M^{(j)}(1)^2)\]
we obtain together with the martingale property of the stochastic integral
\begin{equation*}
\begin{split}
\mathbb{E} \left(\langle \Phi(t)^*x_1,BB^*  \Phi(t)^*y_1 \rangle   \right)
= & \langle x_1, BB^*y_1 \rangle +\mathbb E \left(\int_0^t  \langle \Phi(s)^* A^*x_1,BB^*\Phi(s)^* y_1  \rangle \ ds \right) \\
&+ \mathbb E \left(\int_0^t \langle \Phi(s)^* x_1,BB^*\Phi(s)^* A^*y_1 \rangle \ ds  \right)   \\
&+\sum_{j=1}^{n}\mathbb E \left(\int_0^t  \langle \Phi(s)^*  N_j^*x_1,BB^*\Phi(s)^* N_j^*y_1  \rangle \ ds  \right)\mathbb E\left(M^{(j)}(1)^2\right).
\end{split}
\end{equation*} 
Letting $t$ tend to infinity, we obtain the first Lyapunov equation as by exponential stability
$\lim_{t \rightarrow \infty}\mathbb{E} \left(\left\langle x_1, \Phi(t) \psi_i \right\rangle \left\langle  \Phi(t)\psi_i,y_1 \right\rangle  \right)=0.$

The second Lyapunov equation can be obtained by an analogous calculation:
Let $x_0 \in X$ be arbitrary, then we study the evolution for initial conditions $\sqrt{C^*C}x_0$ in the weak sense of the adjoint flow
\begin{equation*}
\begin{split}
\left\langle  \Psi(t) \sqrt{C^*C}x_0,x_2 \right\rangle= &\left\langle \sqrt{C^*C}x_0,x_2 \right\rangle + \int_0^t \left\langle \Psi(s) \sqrt{C^*C}x_0, Ax_2  \right\rangle \ ds \\ 
&+\sum_{j=1}^n \int_0^t \left\langle \Psi(s) \sqrt{C^*C}x_0, N_jx_2  \right\rangle \ dM^{(j)}_s. 
\end{split}
\end{equation*}
Proceeding as before, stochastic integration by parts yields
\begin{equation*}
\begin{split}
\mathbb{E} &\left(\left\langle x_2, \Psi(t) \sqrt{C^*C}x_0 \right\rangle\left\langle  \Psi(t) \sqrt{C^*C}x_0 ,y_2 \right\rangle  \right)
=  \left\langle x_2, \sqrt{C^*C}x_0\right\rangle\left\langle \sqrt{C^*C} x_0,y_2 \right\rangle \\
&+\mathbb E \left(\int_0^t  \left\langle \sqrt{C^*C}\Psi(s)^*Ax_2,(x_0 \otimes x_0) \sqrt{C^*C}\Psi(s)^*y_2  \right\rangle ds  \right) \\
&+ \mathbb E \left(\int_0^t \left\langle \sqrt{C^*C}\Psi(s)^* x_2,(x_0 \otimes x_0)\sqrt{C^*C}\Psi(s)^*Ay_2 \right\rangle ds  \right)   \\
&+\sum_{j=1}^n \mathbb E \left(\int_0^t  \left\langle \sqrt{C^*C}\Psi(s)^*  N_jx_2,(x_0 \otimes x_0)\sqrt{C^*C}\Psi(s)^* N_jy_2 \right\rangle ds  \right)\mathbb E(M^{(j)}(1)^2).
\end{split}
\end{equation*} 
Using Parseval's identity, i.e. summing over an orthonormal basis replacing $x_0$, yields after taking the limit $t \rightarrow \infty$ the second Lyapunov equation.
\end{proof}
\appendix
\section{Volterra series representation}
\begin{lemm}
\label{Voltconv}
Consider controls $\left\lVert u \right\rVert_{L^2((0,\infty),(\mathbb R^n,\left\lVert \bullet \right\rVert_{\infty}))}  < \frac{\sqrt{2\nu }}{M \Xi}$, an exponentially stable $C_0$-semigroup  $\|T(t)\|\le Me^{-\nu t}$ with $\nu >0,$ and $\Xi:=\sum_{i=1}^{n}\|N_{i}\|$. Then for such control functions the Volterra series
\begin{equation}
\label{eq: volterraseries}
\zeta(t):=\sum_{m=0}^{\infty}\zeta_m(t)
\end{equation}
defined recursively by
\begin{equation*}
 \begin{split}
\zeta_0(t)&:=T(t)\varphi_0, \quad
\zeta_1(t):=\int_0^t T(t-s)\left(\sum_{i=1}^n u_i (s) N_i \zeta_0(s)  +Bu(s) \right) \ ds  \\
\zeta_k(t)&:=\int_0^t T(t-s)\sum_{i=1}^n u_i (s) N_i  \zeta_{k-1}(s) \ ds \text{ for }k \ge 2
 \end{split}
\end{equation*}
converges uniformly on $(0,\infty)$ and the mild solution is given by the Volterra series $\zeta$ \eqref{eq: volterraseries}. 
\end{lemm}
\begin{proof}
For all $k \ge 2$ we obtain recursively an exponentially decreasing bound 
\begin{equation*} \begin{split}
\left\lVert \zeta_k \right\rVert_{L^{\infty}} 
&\le \sup_{t>0}\int_0^t M e^{-\nu(t-s)} \Xi \left\lVert u(s) \right\rVert_{(\mathbb R^n,\left\lVert \bullet \right\rVert_{\infty})} \ ds \ \left\lVert \zeta_{k-1} \right\rVert_{L^{\infty}}  \\
& \le M \Xi \sqrt{\sup_{t>0} \frac{1-e^{-2 t \nu}}{2 \nu}} \left\lVert u \right\rVert_{L^2((0,\infty),(\mathbb R^n,\left\lVert \bullet \right\rVert_{\infty}))} \left\lVert \zeta_{k-1} \right\rVert_{L^{\infty}} \\
&=  \underbrace{\frac{M \Xi}{\sqrt{2 \nu}} \left\lVert u \right\rVert_{L^2((0,\infty),(\mathbb R^n,\left\lVert \bullet \right\rVert_{\infty}))}}_{<1} \left\lVert \zeta_{k-1} \right\rVert_{L^{\infty}} .
\end{split} \end{equation*}
Thus, \eqref{eq: volterraseries} is an absolutely convergent series. To see that $\zeta$ and the mild solution coincide, it suffices to verify that the Volterra series \eqref{eq: volterraseries} satisfies \eqref{mild solution}.
\end{proof}
\section{The composite error system}
\label{sec:comperrsys}
The construction of an auxiliary error system is well-known, see for example \cite{ZL}, and repeated here to explain how to actually compute the error bounds provided in this article. Consider two systems on Hilbert spaces $X$ and $X_r$ with operators (for simplicity we assume $n=1$)
\begin{equation}
\begin{split}
N \in \mathcal L( X), \quad 
A: D(A)\subset X \rightarrow X, \quad C \in \mathcal L(X , \mathcal H), \quad B \in \mathcal L(\mathbb R ,X), \text{ and }\\
N_r \in \mathcal L( X_r), \quad 
A_r: D(A)\subset X_r \rightarrow X_r, \quad C_r \in \mathcal L(X_r , \mathcal H), \quad B_r \in \mathcal L(\mathbb R ,X_r).
\end{split}
\end{equation}
For instance the system on $X$ can be thought of as the full system and the system on $X_r$ as the reduced system.
One can then define a composite error system on the direct sum of Hilbert spaces $\overline{X} = X \oplus X_r$ with the same input space $\mathbb R$ and output space $\mathcal H$ as the initial systems
\begin{equation}
\begin{split}
\label{eq:compsys}
\overline{C} = (C,-C_r), \quad \overline{N} =\left(\begin{matrix} N & 0 \\ 0 & N_r \end{matrix}\right)\, \quad \overline{A} = \left(\begin{matrix} A && 0 \\ 0 && A_r \end{matrix} \right), \quad  \text{ and } \overline{B}= (B,B_r)^T.
\end{split}
\end{equation}
If the composite system satisfies then the stability assumption needed for the gramians to exist, one can then compute to the above composite error system \eqref{eq:compsys} again an observability and reachability gramian $\overline{\mathscr{O}}$ and $\overline{\mathscr {P}}.$ 
Moreover, to the above system there exists an associated Hankel operator which is precisely the difference of the Hankel operators of the two systems, i.e.
\begin{equation}
\overline{H} = \overline{W}\overline{R}=\left\langle \left(\begin{matrix} W \\ -W_r \end{matrix}\right), \left(\begin{matrix} R \\ R_r \end{matrix} \right) \right\rangle = H - H_r=\Delta (H).
\end{equation}
Let $\sigma_i:=\sqrt{\lambda_i(\overline{\mathscr{O}}\overline{\mathscr{P}})}$ be the Hankel singular values of the composite error system indexed by some $i \in I$, then it follows from the unitary invariance of the trace norm that 
\[ \left\lVert \Delta(H) \right\rVert_{\operatorname{TC}} =\left\lVert \overline{H} \right\rVert_{\operatorname{TC}}= \sum_{i \in I} \sigma_i. \]

\smallsection{Acknowledgements} 
This work was supported by the EPSRC grant EP/L016516/1 for the University of Cambridge CDT, the CCA (S.B.). The authors are grateful to Igor Pontes Duff, Keith Glover, and the anonymous referees for inspiring discussions and very useful comments on the manuscript.


%


\end{document}